\documentclass[10pt]{amsart}
\usepackage{amsmath,amssymb,amsfonts} 
\usepackage{graphics}                 
\usepackage[square, numbers]{natbib}
\usepackage[english]{babel}
\usepackage{tikz} 
\usepackage[utf8]{inputenc}
\usepackage[all]{xy}	
\usepackage[T1]{fontenc}
\usepackage{pifont}
\usepackage{tikz-cd} 

\usepackage[margin=1.2in]{geometry}

\newtheorem{Theorem}{Theorem}[section]
\newtheorem{Proposition}[Theorem]{Proposition}
\newtheorem{Definition}[Theorem]{Definition} 
\newtheorem{Remark}[Theorem]{Remark}
\newtheorem{Lemma}[Theorem]{Lemma}
\newtheorem{Corollary}[Theorem]{Corollary}
\newtheorem{Fact}[Theorem]{Fact}
\newtheorem{Example}[Theorem]{Example}
\newtheorem{Notation}[Theorem]{Notation}
\newtheorem{Question}[Theorem]{Question}

\newtheorem{Claim}[Theorem]{Claim}
\newtheorem{Problem}[Theorem]{Problem}

\newcommand{\Z}{\mathbb Z}

\newcommand{\C}{\mathbb C}

\newcommand{\G}{\mathbb G}

\DeclareMathOperator{\tp}{tp}
\DeclareMathOperator{\stp}{stp}
\DeclareMathOperator{\cb}{Cb}
\DeclareMathOperator{\dcl}{dcl}
\DeclareMathOperator{\acl}{acl}

\def\Ind{\setbox0=\hbox{$x$}\kern\wd0\hbox to 0pt{\hss$\mid$\hss}
\lower.9\ht0\hbox to 0pt{\hss$\smile$\hss}\kern\wd0}
\def\Notind{\setbox0=\hbox{$x$}\kern\wd0\hbox to 0pt{\mathchardef
\nn=12854\hss$\nn$\kern1.4\wd0\hss}\hbox to
0pt{\hss$\mid$\hss}\lower.9\ht0 \hbox to
0pt{\hss$\smile$\hss}\kern\wd0}
\def\ind{\mathop{\mathpalette\Ind{}}}

\makeindex

\begin{document}

\title{Notes}
\title{Relative internality and definable fibrations}
\date{\today}
\author{R\'{e}mi Jaoui, L\'{e}o Jimenez, \and Anand Pillay}
\thanks{During the research leading to this paper, as well as its writing, all three authors were supported by NSF grants DMS-1665035 and DMS-1760212. We are very grateful for the support of the NSF, which made this collaboration possible.}

\maketitle

\begin{abstract} 
We first elaborate on the  theory of {\em relative internality} in stable theories from \cite{Jimenez}, focusing on the notion of uniform relative internality (called collapse of the groupoid in \cite{Jimenez}), and relating it to orthogonality, triviality of fibrations, the strong canonical base property, differential Galois theory, and GAGA.  

We prove that $\mathrm{DCF}_0$ does not have the strong canonical base property,  correcting a proof in \cite{Palacin-Pillay}. We also prove that the theory $\mathrm{CCM}$ of compact complex manifolds does 
not have the strong CBP, and initiate a study of the definable Galois theory of projective bundles.  In the rest of the paper we study definable fibrations in $\mathrm{DCF}_0$, where the general fibre is internal to the constants,  including differential tangent 
bundles, and geometric linearizations. We obtain new examples of higher rank types orthogonal to the constants.
\end{abstract}

\tableofcontents

\section{Introduction}
The work in this paper represents the coming together of several  notions and lines of enquiry, which we will relate to each other in various ways.  

One is a  theory of fibrations in the context of stable theories, where the general fibre is internal to some family $\mathcal P$ of partial types over $\emptyset$; namely we have  definable sets $X$ and $Y$ and definable $\pi:X\to Y$  (all over $\emptyset$), and the generic type of a generic fibre $\pi^{-1}(b)$ is {\em internal} to ${\mathcal P}$.  What can be said in terms of the uniformity of the internality data as $b$ varies?  This was the topic of the second author's paper \cite{Jimenez} where such fibrations were studied in terms of an auxiliary definable groupoid $\mathcal G$ and notions of retractability and collapse of the groupoid. 

Another source is Moosa's notion of {\em preservation of internality} \cite{Moosa},  which is again about fibrations $\pi:X\to Y$ and says that if $a$ is a generic point of a generic fibre $\pi^{-1}(b)$ and $\stp(b/c)$ is (almost) ${\mathcal P}$-internal, then so is $\stp(a/c)$. This notion was developed partly to capture in a model-theoretic manner properties of Moishezon morphisms between compact complex manifolds  (where the ambient theory is $\mathrm{CCM}$ and ${\mathcal  P}$ is the projective line).    Chatzidakis, Harrison-Trainor and Moosa prove in \cite{C-H-T-M} that differential jet bundles preserve internality to the constants (in $\mathrm{DCF}_0$). 

Thirdly, is the recent work of  Jin and Moosa \cite{Moosa-Jin}, where the authors characterize when the pullback under the logarithmic derivative, of a set defined by $x' = f(x)$ which is internal to the constants,  is also  internal to the constants   ( $f$ being is a rational function over a field of constants) working in $\mathrm{DCF}_0$. 

Finally, it is the first author's project to effectively describe or determine the semi-minimal analysis of autonomous differential equations on the plane, as well as seeking criteria for the generic type of an algebraic $D$-variet to be orthogonal to the constants in terms of geometric linearization data. In particular the first author and Rahim Moosa
asked for a $D$-variety structure on the plane such that the generic type is orthogonal to the constants but has a (forking) extension which is internal to the constants (and nonalgebraic).  Many such examples are given in this paper. 

Let us now describe the main results of the paper.

In Section 3 we give the basic theory of uniform relative internality (self-contained and independent of \cite{Jimenez}).  What we call relative internality is the situation of  a stationary type $q$ over $\emptyset$ and a $\emptyset$-definable function $\pi$ such that for  $a$ realizing $q$, $\tp(a/\pi(a))$ is stationary and internal to a family $\mathcal P$ of partial types over $\emptyset$.   This situation is witnessed by, for any $a$ realizing $q$, some set $B$ of parameters (independent from $a$ over $\pi(a)$) such that $a$ is in the definable closure of $\pi(a)$, $B$, and some realizations of types in $\mathcal P$. 
The question is to what extent the internality data $B$ depends on $\pi(a)$.  Uniform relative internality means that the internality data can be chosen not depending on the base point $\pi(a)$.  We give various equivalences, as well as relating uniform relative  internality to triviality of the fibration $\pi$  (when $\pi(q)$ is orthogonal to ${\mathcal P}$).  One of the main points is that uniform relative internality means, in effect, that no internality parameters are required, as long as we replace ${\mathcal P}$ by ${\mathcal P}_{int}$, the family of types over $\emptyset$ internal to ${\mathcal P}$.  

We also repeat the observation \cite{Jimenez} that uniform relative internality implies preservation of internality. 

In Section 4, we relate  uniform (almost)  relative internality to the canonical base property, or more accurately the strong canonical base property from \cite{Palacin-Pillay}. We give a correct proof that $\mathrm{DCF}_0$ does not have the strong canonical base property. We also discuss the case of $\mathrm{CCM}$ the many sorted theory of compact complex manifolds. We prove that $\mathrm{CCM}$ does not have the strong canonical base property, via the Galois-theoretic criteria of \cite{Palacin-Pillay}, and ask several questions.

Sections 5, 6 and 7, are devoted to the study of examples in $\mathrm{DCF}_0$, where ${\mathcal P}$ is of course the field of constants.  
In Section 5 ,  we consider the case of ODE's (on the plane) of the form  $x' = f(x)$, $y' = yg(x)$, where $f$, $g$ are rational functions over $\C$, which is more general than the situation studied by Jin and Moosa in \cite{Moosa-Jin} (where $g(x)$ is $1$).  We characterize when the generic type of the total space is orthogonal to the constants. Equivalently, assuming the base $x' = f(x)$ is orthogonal to the constants, we characterize when the fibration is (not) uniformly relatively internal to the constants.

In Section 6,  we study the problem of lifting orthogonality to the constants from a $D$-subvariety $(Z, s|Y)$ of $(X,s)$ to $(X,s)$ itself,  when $Z$ is a hypersurface, using  linearizations of $Z$ and (non) uniform relative internality of corresponding fibrations (over the base $Z$).

In Section 7,  we consider the problem of when differential tangent bundles (of differential algebraic varieties) are uniformly relatively internal (to the constants).  We give some positive results  (for example for differential algebraic groups) but also give a counterexample. Bearing in mind  \cite{C-H-T-M} mentioned above, this gives another example where preservation of internality does not imply uniform relative internality. 

The results of Section 5,6, and 7, give a variety of examples to the question of the first author and Moosa, mentioned above. 

This paper includes fairly pure stability theoretic notions, as well as algebraic geometric notions.
In the preliminaries section we will try to give background and references, including about algebraic $D$-varieties and the interpretation in terms of actual solutions of algebraic differential equations.

We would like to thank Omar Le\' on S\'anchez  and Rahim Moosa for discussions. In particular, Le\'on S\'anchez pointed out a  mistake in a talk on this topic by  the third author at a conference in Kent in July 2019, and his (counter) example is item (ii) in Proposition \ref{uniform-internality-of-tangent}. 

Rahim Moosa insisted on us extracting the proper content and strength of the notion of uniform (almost) relative internality, and this is what we have now done in Section 3. 

In the process of writing this paper, we realized the importance of definable Galois theory, with uniform relative internality corresponding to the special case of triviality of the Galois group  of the generic fibre, at least when ${\mathcal P} = {\mathcal P}_{int}$.  This point of view is reflected especially in Section 4 on the (strong) canonical base property. 

\section{Preliminaries}
Our model theoretic and stability theoretic notation is standard (as in \cite{Pillay-book} for example) except that we use $\mathbb M$ to denote a very saturated model.  All theories $T$ we deal with will be stable and we typically work in $T^{eq}$. This is consistent with our examples  (mainly $\mathrm{DCF}_0$  but also $\mathrm{CCM}$,  which have elimination of imaginaries).  Of course 
stability theory has a rather complicated machinery  (as do other subjects in mathematics) but there are many books and articles available about the basic framework, including the ones referenced in this paper. 

One of the key background notions is {\em internality}: Fix a base set of parameters, say $A$, and a family ${\mathcal P}$ of partial types over $A$ as well as a stationary type $p(x)$ over $A$.  
 Then $p$ is ${\mathcal P}$-internal  (or internal to ${\mathcal P}$) if there is some set $B$ of parameters  (the internality data), such that for some (any) $a$ realizing $p$ which is independent from $B$ over $A$, $a\in dcl(A,B,c)$ for some tuple $c$ of realizations of partial types in ${\mathcal P}$.  If we replace $dcl$ by $acl$ we get the notion of {\em almost internality}. 
There are various equivalences; see Chapter 7 of \cite{Pillay-book}.
This also makes sense when ${\mathcal P}$ is replaced by a family ${\mathcal P}$ of partial types over maybe larger parameter sets but such that the family ${\mathcal P}$ is still invariant under automorphisms fixing $A$ (pointwise). 

We will also be discussing the situation where $Y$ and $P$ are sets definable over $A$, and $Y$ is internal to $P$ meaning that for some set $B$ of parameters containing $A$, $Y\subseteq dcl(P,B)$. This means that $Y$ is in definable bijection with some definable set $Z$ in  $P_{A}^{eq}$. 

We recall canonical bases in stable theories: Given a stationary type $p(x)$ over some set $A$ of parameters, the canonical base of $p$, denoted $Cb(p)$ is the definable closure of the set of $\phi(x,y)$-definitions of the unique global (over $\mathbb M$) nonforking extension $p'$ of $p$, as $\phi(x,y)$ ranges over $L$-formulas (and if $\psi(y,c)$ is the $\phi(x,y)$-definition of $p'$ then we view $\psi(y,c)$ as an element of ${\mathbb M}^{eq}$). 

We will discuss the {\em canonical base property} (CBP) which mostly concerns superstable theories of finite $U$-rank  (i.e. where every type has finite $U$-rank)  but also makes sense for the finite rank part of an infinite rank theory such as $\mathrm{DCF}_0$  (differentially closed fields of characteristic $0$), and can be extended to supersimple theories of finite rank.   In the most general formulation we take ${\mathcal P}$ to be the family of stationary nonmodular types of $U$-rank $1$  (which is $\emptyset$-invariant). The theory $T$ has the canonical base property if whenever  $c = Cb(\stp(a/c))$ then  $\stp(c/a)$ is almost ${\mathcal P}$-internal.

\vspace{5mm}
\noindent
In the same way as we assume familiarity with model theory we will assume familiarity with basic algebraic geometry, as can be found in \cite{Shafarevich} for example.  In many cases the naive point-set point of view suffices, as is described in several model theory texts or papers such as \cite{Poizat-book} or \cite{Pillay-ML}, and where the compatibility with definability is emphasized.  However in Section \ref{linearization} of this paper, the algebraic geometric language is more sophisticated, and the reader is referred to Hartshorne's book \cite{Hartshorne}, in particular section 8 of Chapter II.

When it comes to the combination of algebraic (and  differential) geometry, differential algebra, and model theory, the references tend to be in papers rather than  books, and are less well-known (to general model-theorists or geometers).   A basic text is Marker's article in \cite{Marker} on the model theory of differential fields, specifically the first order theory $\mathrm{DCF}_{0}$ of differentially closed fields of characteristic $0$. For a fairly comprehensive account of the relationship between geometric notions such as algebraic $D$-varieties and definability in differentially closed fields, see Section 1 of \cite{Hrushovski-Itai}. The appendix to \cite{Bertrand-Pillay} also has a translation between differential algebraic and algebraic/analytic language, although mainly in the context of algebraic $D$-groups.

The relation between solutions of algebraic differential equations (as holomorphic, meromorphic, or real analytic functions on a suitable domain)  and differential algebraic varieties, which are more or less definable sets in differentially closed fields, is analogous to the relation between integer or rational solutions of polynomial equations (over $\Z$ say) and the geometric properties of the solution sets in some larger algebraically closed field, such as $\C$  (in characteristic $0$).  The model theory of algebraically closed fields does not give much new information about algebraic varieties, but for reasons we will not go into here, model theory does say a lot about differential algebraic varieties. 

On the model theoretic side we study definable sets in an ambient (maybe saturated) differentially closed field ${\mathcal U}$ with distinguished derivation $\partial$.  On the differential algebra side, we study differential algebraic varieties (DAV's), which are point sets in ${\mathcal U}$ defined by finite systems of differential polynomial equations (in indeterminates $y_{1},..,y_{n}$ say), as well as generalizations to abstract differential algebraic varieties. A DAV is a special case of a definable set, and stability theory brings many new tools to the picture.  On the algebraic geometric side, a key notion is that of an algebraic $D$-variety over a differential field $(K,\partial)$, which was introduced explicitly by Buium.  This is an algebraic variety (say irreducible) over $K$ with an extension of the derivation $\partial$ on $K$ to a derivation ${\partial}'$  of the structure sheaf of $X$.  When $X$ is affine this just means an extension of $\partial$ to a derivation $\partial '$ of the coordinate ring $K[X]$, i.e. a derivation from $K[X]$ to itself.  Either way, we get an extension of the derivation $\partial $ to the function field $K(X)$ of $X$.

Let us discuss briefly the connection between  algebraic $D$-varieties, definable sets, and solutions of differential equations. 
When the underlying variety $X$ of an algebraic $D$-variety is defined over a field $k$ (maybe algebraically closed) of constants, then the $D$-variety structure can be interpreted as a vector field on $X$, namely a regular section $s$ of the tangent bundle $T(X)$ of $X$  (which we assume to also be defined over $k$).  So for each $a\in X$, we have a point $s(a)$ in the tangent space to $X$ at $a$ (where $a$ may live in an arbitrary field containing $k$), and $s$ is a morphism in the sense of algebraic geometry.  What we call $(X,s)^{\partial}$ is the definable (over $k$) set in the universal domain, defined by  $\{a\in X({\mathcal U}):  (a,\partial(a)) = s(a)\}$.  What does it mean in terms of solutions?  Suppose $k = \C$, and 
$K$ is a differential subfield of $\mathcal U$ which contains $k$ and is some field of meromorphic functions on a disc $D$ in $\C$ and moreover on $K$, $\partial$ is just $d/dt$  (where $t$ is the independent variable).  For example $K$ could be simply the field of rational functions $\C(t)$ equipped with $d/dt$. 
Then $X(K)$ is the collection of suitable meromorphic functions from the disc $D$ to the complex variety $X(\C)$.  And the set $(X,s)^{\partial}(K)$ of $K$ points of the (quantifier-free) definable set $(X,s)^{\partial}$ consists of those meromorphic functions $f$ such that $df(1)$ coincides with the vector field $s$  (where $df$ is the differential of $f$ and $1$ is the vector field on the affine (or projective) line corresponding to $d/dt$). 

This geometric interpretation extends to the case when $K$ is the field of meromorphic functions on some open  subset $U$ of a complex algebraic variety $V$, which is equipped with a vector field. 

We will be discussing invariant subvarieties $Y$ of an algebraic $D$-variety $(X,s)$, all defined over the constants $\C$ say.  This just means that the vector field $s$ on $X(\C)$ restricts to a vector field on $Y(\C)$, namely for each $a\in Y$, $s(a)$ is in the tangent space to $Y$ at $a$.  When $Y$ is a curve, this is another kind of solution to $(X,s)$, namely an integral curve. 

Although we will largely be concerned with algebraic $D$-varieties over fields of constants, it is worth mentioning the geometric interpretation for $D$-varieties not necessarily defined over the constants.  For simplicity we consider the case of an algebraic $D$ variety $(X,\partial ')$ defined over $(\C(t), d/dt)$.  
So $X$ is the generic fibre of a dominant rational map $\pi$ from a complex variety $\mathcal X$ to the affine line, with a lifting of the vector field $\bf 1$ on the affine line to a rational vector field $\bf s$ on $\mathcal X$. (Namely such that $d\pi({\bf s}) = {\bf 1}$.)   Then  points of $(X,\partial')^{\partial}(\C(t))$ consists of 
rational {\em sections} $f$ of $\pi$ such that  $df({\bf 1}) = {\bf s}$ in the obvious sense.   Likewise if we look for points of $(X,\partial ')(K)$ in some field $K$ of meromorphic functions on some disc $D$. And one has a similar treatment when $(\C(t),d/dt)$ is replaced by some complex algebraic variety equipped with a vector field.

In any case this situation of $\pi:{\mathcal X} \to {\mathbb A}_{1}$ with a lifting of the vector field ${\bf 1}$ to a vector field ${\bf s}$ is also called an Ehresmann connection. 

Let us mention in passing that in the situation of an algebraic $D$-variety $(X,\partial ')$ over a possibly nonconstant differential field,  $\partial '$ is represented by a regular or rational section $s$ not of the tangent bundle $T(X)$ by a shifted tangent bundle $T(X)_{\partial}$.  See \cite{Hrushovski-Itai}. 

At the end of section 4 we have  a relatively self contained discussion of the theory of compact complex manifolds. We will give relevant references there.

\section{Uniform Relative Internality}

This section will mostly be pure model theory. In later sections we will take a look at concrete examples.

We work in the context of a complete stable theory $T = T^{eq}$, and we will be assuming for convenience that $acl(\emptyset) = dcl(\emptyset)$ (i.e. all complete types over $\emptyset$ are stationary). We also fix an ambient monster model $\mathbb{M}$.

As discussed above ${\mathcal P}$ denotes a set of partial types over $\emptyset$, sometimes identified notationally with the union of the sets of realizations of types in ${\mathcal P}$ in $\mathbb M$. Moreover $q$ denotes a complete type over $\emptyset$ and $\pi$ a $\emptyset$-definable function whose domain includes the realizations of $q$. We assume that $\tp(a/\pi(a))$ is stationary for some/any realization $a$ of $q$. 

Let $\pi(q)$ denote $\tp(\pi(a))$ for $a$ some (any) realization of $q$.  For $b$ realizing  $\pi(q)$, we let $q_{b}$ denote $\tp(a/b)$ for a some (any) realization of $q$ such that $\pi(a) = b$. We will also often denote this $q_{\pi(a)}$. 

Of course we could work instead over some algebraically closed set $A$ of parameters, and all our results would hold. We will stick with $A= \emptyset$ to ease notation, at least in the general theory.

\begin{Definition} We say that $(q,\pi)$ is relatively ${\mathcal P}$-internal (or relatively internal to ${\mathcal P}$) if for some (any) $a$ realizing $q$, $\tp(a/\pi(a))$ is internal to ${\mathcal P}$, namely for some tuple of parameters $e$ independent from $a$ over $\pi(a)$, $a\in dcl(\pi(a),e,{\mathcal P})$.  $(q,\pi)$ is said to be relatively {\em almost} ${\mathcal P}$-internal if the above holds with $dcl$ replaced by $acl$.

\end{Definition}

The following is well-known (see Section 4, Chapter 7, of \cite{Pillay-book}.)

\begin{Remark}  (i) $(q,\pi)$ is relatively ${\mathcal P}$-internal iff  for any $b$ realizing $\pi(q)$, there is a tuple $e$ such that for any $a$ realizing $q$ with $\pi(a) = b$, $a\in dcl(b,e,{\mathcal P})$.  
\newline
(ii) Moreover, in (i), we can choose $e$ to be a finite independent (over $b$) tuple of realizations of $q_{b}$. 
\newline
Likewise for {\em almost} in (i) and (ii). 
\end{Remark} 

The notion of uniform relative internality concerns if and when  the parameter $e$ can be chosen not to depend on the point $b$ realizing $\pi(q)$.  We will take as a precise definition the analogue of Definition 2.1 above.  

\begin{Definition}\label{unif-int-definition}
The pair $(q,\pi)$ is uniformly relatively (resp. almost) ${\mathcal P}$-internal if there is a finite tuple $e$ of parameters such that for some (any) $a$ realizing $q$ with  $a$ independent from $e$ over $\emptyset$, we have $a\in dcl(\pi(a),e, {\mathcal P})$  (resp. $a\in acl(\pi(a), e, {\mathcal P})$). 
\end{Definition}

Any (almost) $\mathcal{P}$-internal type $\tp(a/b)$ gives rise to a relatively (almost) $\mathcal{P}$-internal pair $(\tp(ab/\emptyset),\pi)$, where $\pi$ is the projection on the $b$ coordinate. As we will frequently need to consider such a situation, we define the following for convenience:

\begin{Definition}

The stationary type $\tp(a/b)$ is said to be uniformly (almost) $\mathcal{P}$-internal if the pair $(\tp(ab/\emptyset),\pi)$, where $\pi$ is the projection on the $b$ coordinate, is relatively uniformly (almost) $\mathcal{P}$-internal.

\end{Definition}

Of course a much stronger condition is that no parameter $e$ is needed. Compare to the notion in \cite{Moosa} of ${\mathcal P}$-algebraicity. 

\begin{Definition}  The pair $(q,\pi)$ is relatively  ${\mathcal P}$-definable, if for some (any) $a$ realizing $q$ there is some tuple $c$  from ${\mathcal  P}$ such that $a\in dcl(\pi(a),c)$. 
\newline
The pair $(q,\pi)$ is relatively  ${\mathcal P}$-algebraic if for some (any) $a$ realizing $q$, $a\in acl(\pi(a),c)$ for some tuple $c$ from ${\mathcal P}$. 
\end{Definition}  

\begin{Remark}  One easily sees that $(q,\pi)$ relatively ${\mathcal P}$-definable implies $(q,\pi)$ uniformly relatively  ${\mathcal P}$-internal implies $(q,\pi)$ relatively ${\mathcal P}$-internal. And likewise with algebraic and almost.  The converses of these implications will not always hold. See below. 

\end{Remark}

Let us start by stating a few elementary characterizations of uniform (almost) internality, in analogy with the situation for (almost) internality: 

\begin{Lemma} The following are equivalent:
\begin{itemize}
\item[(1)] $(q,\pi)$ is uniformly relatively ${\mathcal P}$-internal (in the sense of Definition \ref{unif-int-definition}).
\item[(2)] $(q,\pi)$ is uniformly relative internal to ${\mathcal P}$, witnessed by $e$ an independent tuple of realizations of $q$. 
\item[(3)] There is a finite tuple $e$ of independent realizations of $q$ such that whenever $a$ realizes $q$ and $\pi(a)$ is independent from $e$ over $\emptyset$ then $a\in dcl(\pi(a),e,{\mathcal P})$. 
\item[(4)] There is a (possibly infinite) tuple $e'$ of independent realizations of $q$ such that for some (any) realization $a$ of $q$, $a\in dcl(\pi(a),e',{\mathcal P})$. 
\end{itemize}
\end{Lemma} 
\begin{proof} (1) implies (2).  Assume (1) and let $e$ be a tuple witnessing uniform internality, namely there is $a$ realizing $q|e$, and $a\in dcl(\pi(a),e,{\mathcal P})$. 

Let $M$ be a saturated model containing $e$ such that $a$ realizes $q|M$.  Let $c$ be a tuple of realizations of types from ${\mathcal P}$ such that $a\in dcl(\pi(a), c, e)$.  Let  $(a_{i},c_{i})_{i \in \mathbb{N}}$ be a Morley sequence in $\stp(a,c/e)$ contained in $M$. Then $\tp(a,c/M)$ is definable over $(a_{i},c_{i})_{i \in \mathbb{N}}$ whereby $a\in dcl(\pi(a),c,a_{1},c_{1},...,a_{n},c_{n})$ for some $n$.

Thus $(a_{1},..,a_{n})$ is an independent tuple of realizations of $a$, independent from $a$ over $\emptyset$ and $a\in dcl(\pi(a),(a_{1},..,a_{n}),{\mathcal P})$, which yields (2). 
\newline
(2) implies (3).  Let $e$ be a  tuple of independent realizations of $q$ given by (2), namely for any $a$ realizing $q|e$ (i.e. $a$ independent from $e$), $a\in dcl(\pi(a),e,{\mathcal P})$.  Now let $b$ realize $\pi(q)|e$, i.e. $b$ is independent from $e$.  We know that $q_{b}$ is ${\mathcal P}$-internal. So by Lemma 4.2, Chapter 7 of \cite{Pillay-book}, there is a sequence ${\bar\alpha}$ of realizations of $q_{b}$ such that any realization $a$ of $q_{b}$ is in the definable closure of $b, {\bar\alpha}$ and ${\mathcal P}$. We may assume that ${\bar\alpha}$ is independent from $e$ over $b$. But then ${\bar{\alpha}}$ is independent from $e$ over $\emptyset$, hence by (2) any $\alpha\in {\bar\alpha}$ is in $dcl(b,e,{\mathcal P})$.  It follows that any realization of $q_{b}$ is in $dcl(b,e,{\mathcal P})$. 
\newline
(3) implies (4).  Let $e$ be given by (3) and $e' = (e_{i})_i$ be a Morley sequence of length $|T|^{+}$ in $\tp(e)$.
Consider $a$ an arbitrary realization of $q$. Then  $b = \pi(a)$ is independent from some $e_{i}$, so by (3) (and automorphism), we obtain $a\in dcl(b,e_{i},{\mathcal P})$. 
Hence for any realization $a$ of $q$, we have $a\in dcl(b,e',{\mathcal P})$.
\newline
(4) implies (1). The only thing to do is to replace the possibly infinite tuple $e'$ by a finite subtuple.  Let $a$ realize $q$ such that $a$ is independent of $e'$. By (4) $a\in dcl(\pi(a),e',{\mathcal P})$, hence there is a finite subtuple $e$ of $e'$ and a $\emptyset$-definable function $f$ such that $a = f(\pi(a),e,c)$ for some finite tuple $c$ of realizations of types from ${\mathcal P}$. So for any other realization $a'$ of $q$ which is independent from $e$, $\tp(a,e) = \tp(a',e)$ by stationarity, therefore  $a' = f(\pi(a'),e,c')$ for some suitable $c'$ from ${\mathcal P}$, and we have (1). 
\end{proof}

\begin{Remark} (i)  If either $q$ has finite weight, or ${\mathcal P}$ is a collection of formulas (rather than partial types), then we can in (4) choose $e'$ to be a finite tuple.  The first is contained in Proposition 5.9 of \cite{Jimenez} and the second is a compactness argument. 
\newline
(ii)  Lemma 3.7  goes through with obvious modifications for {\em uniformly almost relatively internal}. 
\end{Remark}

We now aim towards \ref{unif-int-iff-Pint-alg} below which says that a slight tweaking of the strong properties of ${\mathcal P}$-definability and algebraicity will make them {\em equivalent} to uniform relative ${\mathcal P}$-internality, and uniform relative almost ${\mathcal P}$-internality, respectively.

\begin{Definition}

Let $\mathcal{P}$ be a family of partial types over $\emptyset$. The internal closure of $\mathcal{P}$, denoted $\mathcal{P}_{\mathrm{int}}$, is the set of $\tp(a/\emptyset)$ such that $\stp(a/\emptyset)$ is $\mathcal{P}$-internal.

\end{Definition}

We have:

\begin{Proposition}\label{unif-int-iff-Pint-alg}

The following are equivalent:

\begin{itemize}
    \item[(1)] $(q,\pi)$ is uniformly relatively $\mathcal{P}$-internal
    \item[(2)] $(q,\pi)$ is uniformly relatively $\mathcal{P}$-internal, witnessed by a tuple $t$ such that $\tp(t/\emptyset)$ is $\mathcal{P}$-internal
    \item[(3)] For some (any) realization $a$ of $q$, the type $\tp(a/\pi(a))$ is $\mathcal{P}_{\mathrm{int}}$-definable.
\end{itemize}
And the same equivalences hold after adding almost in (1) and (2), and writing algebraic in place of definable in (3). 

\end{Proposition}

\begin{proof}

(1) implies (2). By assumption, there are $a \models q$, a tuple $e$, independent from $a$ over $\emptyset$, and $c \in \mathcal{P}$, such that $ a \in dcl(\pi(a), e, c)$. Consider $t = Cb(\stp(ac/e))$. 

By properties of canonical bases, we have $ac \ind_t e$, thus $a \in dcl(t,\pi(a),c)$. As $t \in acl(e)$, we also know that $a$ is independent from $t$ over $\emptyset$. Since $q$ is stationary, for any $b \models q$ independent from $t$ over $\emptyset$, we have $b \in dcl(t,\pi(b),\mathcal{P})$. 

As for internality of $\tp(t)$, recall that $t \in dcl((a_ic_i)_{i = 1 \cdots n})$, for $(a_ic_i)_{i = 1 \cdots n}$ a Morley sequence in $\stp(ac/e)$. Because $a \ind e$, we can prove, by induction and forking calculus, that $a_1 \cdots a_n \ind e$, and therefore $a_1 \cdots a_n \ind t$. Lastly, since $\tp(c_i/\emptyset) = \tp(c/\emptyset)$ for all $i$, we see that $c_i \in \mathcal{P}$ for all $i$. This, combined with the independence previously obtained, yields  $\mathcal{P}$-internality of $\tp(t/\emptyset)$.

To replace $t$ by an internal tuple, recall the following general fact: if $\tp(t/\emptyset)$ is almost $\mathcal{P}$-internal, there is $d \in dcl(t)$ such that $\tp(d/\emptyset)$ is $\mathcal{P}$-internal and $t \in acl(d)$. Replacing the previously obtained $t$ by $d$, we get the result.
\newline
(2) implies (3): Let $a$ be a realization of $q$, and $t$ as in (1). There is a tuple $c$ of realizations of $\mathcal{P}$ such that $a \in dcl(\pi(a),t,c)$. As the  type $\tp(c,t/\emptyset)$ is in $\mathcal{P}_{\mathrm{int}}$,  $\tp(a/\pi(a))$ is $\mathcal{P}_{\mathrm{int}}$-definable. A straightforward automorphism argument yields the same for any other realization of $q$.
\newline
(3) implies (1): Let $a$ be such that $\tp(a/\pi(a))$ is $\mathcal{P}_{\mathrm{int}}$-algebraic. There is a tuple $c$ of realizations of $\mathcal{P}_{\mathrm{int}}$ such that $a \in dcl(\pi(a),c)$. 

By definition of $\mathcal{P}_{\mathrm{int}}$, there is a tuple $e$ such that $c \in dcl(e,\mathcal{P})$ and $e \ind c$. By choosing $e$ independent from $a$ over $c$, we have that $e \ind a$. As $a \in dcl(\pi(a),e,\mathcal{P})$, thus $(q,\pi)$ is uniformly relatively $\mathcal{P}$-internal.

The same arguments work for the variation  in the last sentence of the Proposition. 

\end{proof}

One of the simplest ways for a  fibration $(q, \pi)$ to be uniformly $\mathcal{P}$-internal is to be (naturally) interdefinable with a product of $\pi(q)$ and a $\mathcal{P}$-internal type.  More precisely:

\begin{Definition}

The internal fibration $(q,\pi)$ is called:

\begin{itemize}
    \item Trivial if there is $a \models q$ and $c \in dcl(a)$ such that $\pi(a) \ind c$ and $a \in dcl(\pi(a),c)$, with $r = \tp(c/\emptyset)$ internal to $\mathcal{P}$. Thus $q$ is in definable bijection with $\pi(q) \otimes r$.
    \item Almost trivial if there is $a \models q$ and $c \in \acl(a)$ such that $\pi(a) \ind c$ and $a \in acl(\pi(a),c)$, with $\tp(c)$ internal to $\mathcal{P}$. Thus $q$ is interalgebraic with $\pi(q) \otimes r$.
\end{itemize}

\end{Definition}

In some cases, this is the only way to obtain uniform (almost) internality.

\begin{Definition}

Fix a family of partial types $\mathcal{P}$ over $\emptyset$. A stationary type $p(x)$ (over $\emptyset$ say) is:
\begin{itemize} 
    \item  nonorthogonal to ${\mathcal P}$ if for some set $A$ of parameters, and some (any) realization $a$ of $p|A$, there is some tuple $c$ of realizations of types in ${\mathcal P}$ such that $a$ forks with $c$ over $A$. 
    \item non weakly orthogonal to ${\mathcal P}$ if we can choose $A = \emptyset$ above. Namely for some (any) realization $a$ of $p$, there is a tuple $c$ from ${\mathcal P}$ such that $a$ forks with $c$ (over $\emptyset$).
\end{itemize}
  
Note that nonorthogonality of a complete stationary type $p$ to a family ${\mathcal P}$ of partial types is often ``called $p$ is not foreign to ${\mathcal P}$" in the literature (such as \cite{Pillay-book}).

\end{Definition}

\begin{Fact} In the context of Definition  3.12, $p(x)$ is nonorthogonal to ${\mathcal P}$ if and only if $p(x)$ is non weakly  orthogonal to ${\mathcal P}_{int}$.
\end{Fact} 
\begin{proof} Right implies left is almost immediate.
For left implies right. Suppose $p= \tp(a/\emptyset)$ is nonorthogonal to ${\mathcal P}$.  By Corollary 4.6 of Chapter 7 in \cite{Pillay-book}, there is $c\in dcl(a)\setminus acl(\emptyset)$ such that $\tp(c/\emptyset) \in {\mathcal P}_{int}$. So $a$ forks with $c$ over $\emptyset$ and we are finished.

\end{proof}

We have the following which will be behind many of the applications in subsequent sections: 

\begin{Proposition} \label{equivalence generale}  Suppose that $\pi(q)$ is orthogonal to ${\mathcal P}$. Then the following are equivalent:
\begin{enumerate}
\item  $(q,\pi)$ is uniformly relatively almost ${\mathcal P}$-internal,
\item $(q,\pi)$ is almost trivial.
\end{enumerate}
Moreover if for some (any) $a \models q$, the type $\tp(a/\pi(a))$ has $U$-rank one, then (1) and (2) are also equivalent to

(3)  $q$ is nonorthogonal to ${\mathcal P}$.

\end{Proposition}

\begin{proof}

The (2) implies (1) direction is fairly immediate and does not need the orthogonality of $\pi(q)$ to ${\mathcal P}$: to be more formal, suppose that $a\models  q$ is interalgebraic  with $(\pi(a),c)$ where $\tp(c) \in {\mathcal P}_{int}$, and $\pi(a)$ is independent from $c$ over $\emptyset$. Using the fact that $\tp(c)$ is stationary and ${\mathcal P}$-internal, we can choose $A$ independent from $(a,c)$ over $\emptyset$ such that $c\in dcl(A,d)$ where $d$ is a tuple of realizations of types in ${\mathcal P}$.  As $a\in acl(\pi(a),c)$ we see that $a\in acl(\pi(a),A, d)$ so by Definition 3.3, $(q,\pi)$ is uniformly relatively almost ${\mathcal P}$-internal.

 We now prove (1) implies (2). 

By Proposition \ref{unif-int-iff-Pint-alg}, there is $a \models q$ and a tuple $t$ of realizations of $\mathcal{P}_{\mathrm{int}}$ such that $a \in \acl(\pi(a),t)$. Consider $s = \cb(\stp(t/a))$, we have that $s \in \acl(a)$. As $t \ind_{s} a$, we have that $a \in \acl(\pi(a),s)$. Moreover, we know that $s \in \dcl((t_i)_{i \in \mathbb{N}})$, for some Morley sequence $(t_i)_{i \in \mathbb{N}}$ in $\stp(t/a)$, thus $\tp(s)$ is  $\mathcal{P}$-internal, and $\pi(a) \ind s$ by the orthogonality assumption. 

Now we deal with the moreover part. First assume that $(q,\pi)$ is almost trivial. Let $a \models q$ and $c$ witness almost triviality, namely $\tp(c)\in {\mathcal P}_{int}$, $c$ is independent with $\pi(a)$ and $a$ is interalgebraic with $(\pi(a),c)$.  In particular $c\in acl(a)\setminus acl(\emptyset)$, so $a$ forks with $c$, and clearly $q$ is nonorthogonal to $\mathcal P$. 
(Note that we do not need here the assumptions that $\pi(q)$ is orthogonal to ${\mathcal P}$ or $\tp(a/\pi(a))$ is $U$-rank $1$, but only that $\tp(a/\pi(a))$ is not algebraic.)

Secondly, assume that $q$ is nonorthogonal to $\mathcal P$. Let $a$ realize $q$. By Lemma 4.6 in Chapter 7 of \cite{Pillay-book}, there is $c\in dcl(a)\setminus acl(\emptyset)$ such that $\tp(c)\in {\mathcal P}_{int}$.  By our assumptions $c$ is independent with $\pi(a)$, and thus as $\tp(a/\pi(a))$ has $U$-rank $1$, $a\in acl(\pi(a), c)$. So $a$ and $(\pi(a), c)$ are interalgebraic, giving almost triviality. 

\end{proof}

It is interesting to compare uniform internality with other strengthenings of internality, namely algebraicity and preservation of internality: 

\begin{Definition}

The type $\tp(a/b)$ is said to preserve $\mathcal{P}$ internality if it is stationary and whenever $c$ is a tuple such that $\stp(b/c)$ is almost $\mathcal{P}$-internal, then so is $\stp(a/c)$.

\end{Definition}

We will now restate (and reprove) Proposition 5.17 of \cite{Jimenez} in our current vocabulary:

\begin{Proposition}\label{unif_int_implies_preserve}
Suppose $\tp(a/b)$ is uniformly almost $\mathcal{P}$-internal, then $\tp(a/b)$ preserves internality to ${\mathcal P}$.
\end{Proposition}

\begin{proof}  This is basically immediate from \ref{unif-int-iff-Pint-alg}. For the assumptions give $e\in {\mathcal P}_{int}$ such that $a\in acl(b,e)$.   Let $c$ be such that $\stp(b/c)$ is almost ${\mathcal P}$-internal. Then working over $c$ we have almost ${\mathcal P}$-internality of $\stp(b)$ and $\stp(e)$. As $a\in acl(b,e)$, also $\stp(a)$ is almost ${\mathcal P}$-internal. 
\end{proof} 

It is natural to ask about the reverse implication in Proposition \ref{unif_int_implies_preserve} and we will give several counterexamples in the next sections.  In fact giving such counterexamples of varying kinds is among the themes of this paper.

Also remark that $\mathcal{P}$-algebraicity is strictly stronger than uniform almost $\mathcal{P}$-internality (and likewise with $\mathcal P$-definability and without the almost). It is easy to produce examples. Here is one. Work in the theory $\mathrm{DCF}_0$ of differentially closed fields. Let $k$ be an algebraically closed field of constants (over which we work). Let $d\in k$ be nonzero, and let $a$, $b\neq 0$ in the differential closure of $k$ be such that $a' = d$ and $b'/b = d$. Let $q = \tp(a,b/k)$ and $\pi$ be the projection on the second coordinate.  So $q$ is internal to the constants, whereby $(q,\pi)$ is uniformly relatively internal to the constants.  The (differential) Galois theory tells us that $a$ and $b$ are independent over $k$, hence $a$ is not in the algebraic closure of $k,b$ and some constants (as all constants in the differential closure of $k$ are in $k$).

\section{The canonical base property}

In this section we discuss the relationship of uniform relative almost internality to the canonical base property, in particular a more restrictive notion called the {\em strong} canonical base property in \cite{Palacin-Pillay}. 

The canonical base property, as formulated in \cite{Moosa-Pillay}, concerns superstable theories $T$ where every type has finite $U$-rank, which we abbreviate by  {\em $T$ is superstable of finite rank.} This could make sense for an arbitrary stable theory $T$ by considering the many-sorted structure consisting of $\emptyset$-definable sets of in which every type has finite $U$-rank.  There are also formulations for simple theories, etc... But let us for now focus on the \cite{Moosa-Pillay} context of $T$ superstable of finite rank, although we will subsequently make some additional restrictions.  We will also use \cite{Moosa-Pillay} as a reference for the notions being discussed here. 

We will start by letting ${\mathcal P}$ denote the family of non locally modular types of $U$-rank $1$.  This is a $\emptyset$-invariant family, but is NOT a family of types over $\emptyset$, so does not yet fit into the context of the previous section.  

\begin{Definition}
$T$ (superstable of finite rank)  is said to have the canonical base property (CBP) if whenever $a,b$ are tuples in the monster model such that $b = \mathrm{Cb}(\stp(a/b))$, we have that $\tp(b/a)$ is almost $\mathcal{P}$-internal. 

\end{Definition}

The discovery of this notion was very influenced by results in complex geometry. See \cite{Pillay-Campana} for the background. In \cite{Pillay-Ziegler}  the CBP was proved for the finite rank part of $\mathrm{DCF}_0$, yielding a quick account of function field Mordell-Lang in characteristic $0$.  In \cite{Hrushovski-Palacin-Pillay} a theory $T$ without the CBP was given. 


The uniform canonical base property (UCBP) was first formulated in \cite{Moosa-Pillay}, influenced again by complex geometric considerations, in particular the notion of a Moishezon morphism.   Here is a slight strengthening, which was proved by Chatzidakis \cite{Chatzidakis} to be {\em equivalent} to the CBP.

\begin{Definition}\label{uniform-internality}

$T$ (superstable of finite rank) has the uniform canonical base property (UCBP) if for any tuple $a,b$ such that $b = \mathrm{Cb}(\stp(a/b))$ then $\tp(b/a)$ preserves $\mathcal{P}$-internality, in the sense that for any $c$ if $\stp(a/c)$ is $\mathcal P$-internal so is $\stp(b/c)$. 

\end{Definition}





\subsection{The strong CBP} Given the material in the previous section, it is natural to ask to ask whether in place of $\tp(b/a)$ preserving $\mathcal{P}$-internality in Definition \ref{uniform-internality}, one could ask for  $\tp(b/a)$ to be uniformly ${\mathcal P}$-internal.   However, according to our conventions in Section 3, this only makes sense when ${\mathcal P}$ is a collection of partial types over $\emptyset$. This is why we will pass to a more restrictive context, as described in \cite{Palacin-Pillay}. 




\begin{Notation}

Let $T$ be superstable theory of finite rank. We say that $T$ satisfies $(*)$ if any non locally modular stationary $U$-rank one type is nonorthogonal to $\emptyset$.

If $T$ satisfies $(\ast)$, we denote by $\mathcal Q$  the family of types consisting of $\tp(a/\emptyset)$, where $a \in \mathbb{M}$ and $\stp(a/\emptyset)$ is internal to the family of non locally modular $U$-rank one types. 
\end{Notation}

So $T$ satisfies $(\ast)$ if for every non locally modular $U$-rank one type $p$, there is a type $q$ over $\emptyset$ such that $p$ and $q$ are non orthogonal. In \cite{Palacin-Pillay}, this condition is also called nonmultidimensional with respect to non locally modular rank one types.

What did we gain? The family $\mathcal Q$ is a family of types over $\emptyset$ while $\mathcal P$ is only an $\emptyset$-invariant family. Moreover, we have (see  Lemma 3.3 of \cite{Palacin-Pillay}):
%

\begin{Proposition}

If $T$ satisfies $(*)$, then almost internality to $\mathcal{P}$ (the family of non locally modular types of $U$-rank $1$) is equivalent to almost internality to $\mathcal{Q}$.

\end{Proposition}
%
%
%
%
%

The following definition,  is from \cite{Palacin-Pillay}:

\begin{Definition}
(Assume $T$ is superstable of finite rank satisfying (*).)  $T$ has the strong canonical base property (strong CBP)  if for any $a,b$ with $\tp(a/\emptyset)$ finite and $b = \mathrm{Cb}(\stp(a/b))$, we have that $\tp(b/a)$ is $\mathcal{Q}$-algebraic.
\end{Definition}

One of the reasons for introducing the strong CBP was to get an exact and robust (under naming parameters) equivalence to the rigidity of Galois groups relative to ${\mathcal Q}$.

%
%

We will be working with definable automorphism groups or definable Galois theory.  We refer the reader to the introduction to \cite{Palacin-Pillay} and Fact 1.1.  Roughly speaking if ${\mathcal Q}$ is a family of partial types over $\emptyset$ and $p$ is a stationary type over $\emptyset$ which is internal to ${\mathcal Q}$ (in say an ambient stable theory) then the group of permutations of the set of realizations of $p$ induced by automorphisms of $\mathbb M$ which fix  (the collection of realizations of types in) ${\mathcal Q}$ pointwise is isomorphic (as a group action) to a type-definable over $\emptyset$-group $G$ with $\emptyset$-definable action on the set of realizations of $p$. 
We could of course work over a set $A$ of parameters instead of $\emptyset$. 
 We call $G$ the Galois group of $p$ relative to ${\mathcal Q}$.

When $T$ is a superstable finite rank theory satisfying (*) and $\mathcal Q$ is as defined earlier (types over $\emptyset$ internal to the family of non locally modular rank $1$ types), then  one of the main theorems of \cite{Palacin-Pillay} is:
\begin{Theorem} \label{strongCBP-rigidity} Let $T$ be a theory satisfying $(\ast)$. $T$ has the strong CBP iff whenever $p$ is a stationary type over a set $A$ of parameters which is $\mathcal Q$-internal, and $G$ is the Galois group of $p$ relative to ${\mathcal Q}$, then $G$ is {\em rigid}, namely every connected type-definable subset of $G$ is type-defined over $acl(A)$. 
\end{Theorem}

Remark that the rigidity of a type-definable group doesn't depend on the choice of the set $A$ of parameters over which it is type-defined. Similarly, a point made in \cite{Palacin-Pillay} is that a theory $T$ has the strong CBP iff $T_{A}$  has the strong CBP where $T_{A}$ is the theory of ${\mathcal M}$ with constants (names) for a subset $A$ of $\mathcal M$.

\begin{Definition}
Let $\mathcal P^0$ be any family of types over $\emptyset$ and $(q,\pi)$ a relatively $\mathcal P^{0}$-internal pair.
We say that a pair $(q,\pi)$ is rigid over $\mathcal P^{0}_{int}$  if the Galois group of $q_{\pi(a)}$ relative to $\mathcal P^{0}_{int}$ is rigid.
\end{Definition}

\begin{Remark} 
Let $\mathcal P^0$ be any family of types over $\emptyset$ and $(q,\pi)$ a relatively $\mathcal P^{0}$-internal pair, then $(q,\pi)$ is uniformly $\mathcal P^{0}$-internal (resp. almost $\mathcal P^{0}$-internal) if and only if the Galois group of $q_{\pi(a)}$ relatively to $\mathcal P^{0}_{int}$ is trivial (resp. finite).
\end{Remark}

So if $(q,\pi)$ is uniformly  almost $\mathcal P^0$-internal then it is rigid over $\mathcal P^0_{int}$.
Here, ${\mathcal Q}_{int} = {\mathcal Q}$ (see Remark 4.3 of Chapter 7 of \cite{Pillay-book})  so that Theorem \ref{strongCBP-rigidity} reads as: the failure of the strong CBP is witnessed by the existence of relatively $\mathcal P$-internal pairs $(q,\pi)$ (over $\emptyset$) satisfying a strong form of non-uniform relative internality, namely which are not rigid over $\mathcal Q$.

So far in this section we have just been playing around with definitions and equivalent formulations.  But we will now try to say something substantial by constructing pairs $(q,\pi)$ which are not rigid over $\mathcal Q$ in $\mathrm{DCF}_0$ and $\mathrm{CCM}$. We start with the case of (the finite rank part of) $\mathrm{DCF}_0$, fixing a gap in \cite{Palacin-Pillay} where an argument going from Galois theory relative to the constants ${\mathcal C}$ to Galois theory relative to ${\mathcal C}_{int} = \mathcal Q$ is missing.

\subsection{$\mathrm{DCF}_0$}

So $T$ is (the finite rank part of) $\mathrm{DCF}_0$, which is totally transcendental whereby all type-definable groups are definable.  We will assume familiarity with differential Galois theory, which is the theory of definable automorphism groups in the special case of $\mathrm{DCF}_0$. The paper \cite{Leon-Sanchez-Pillay} has a discussion of differential Galois theory which includes definitions and details around the material below  (Picard-Vessiot, strongly normal, extensions of differential fields). 

\begin{Proposition} There exists a relatively internal $(q,\pi)$ in $\mathrm{DCF}_0$ with Galois group over ${\mathcal C}_{int} = \mathcal Q$   (definably isomorphic) to $PGL(2, \mathcal C)$ (working over some field $K$).  In particular, (the finite rank part of) $\mathrm{DCF}_0$ does not have the strong $\mathrm{CBP}$.
\end{Proposition}

Note that  $PGL(2,{\mathcal C})$ is not rigid. For example there  are infinitely many Borel subgroups (maximal connected solvable) and these are all conjugate.  So in particular $(q,\pi)$ is not rigid over $\mathcal Q$. Hence the second part of the statement follows from the first part and Theorem \ref{strongCBP-rigidity}.

\begin{proof}
Work in a saturated (if one wishes) model $\mathbb U$ of $\mathrm{DCF}_0$, with field of constants ${\mathcal C}$. Let  $C$ be some small algebraically closed field of constants such as $\mathbb C$.  Let $K = \C(t)$, where $\partial(t) = 1$.  It is well-known that there exists a Picard-Vessiot extension $L$ of $K$ such that  $Aut_{\partial}(L/K)$ is  $PGL(2,\C)$.  What this means is that $L$ is contained in the prime model over $K$ (whose field of constants is $\C$), $L$ is generated 
by a tuple $b$ such that $\stp(b/K)$ is internal to ${\mathcal C}$ in the strong form that $b_{1}\in dcl(K,b,{\C})$ whenever $b_{1}$ has the same type as $b$ over $K$, and the Galois group of $\stp(b/K)$ with respect to ${\mathcal C}$ is definably (in $\mathcal U$) isomorphic to $PGL(2,{\mathcal C})$. Restricted to the prime model over $K$ this gives the isomorphism of $Aut_{\partial}(L/K)$ with $PGL(2,\mathbb{C})$.  

We set $q = \tp(b,t/\mathbb{C})$ and $\pi$ the projection on the $t$-coordinate. By construction, $(q,\pi)$ is relatively $\mathcal C$-internal and $q_{\pi(a)}  = \tp(b/K)$.

\begin{Claim}
 $PGL(2,{\mathcal C})$ is also (definably isomorphic to) the Galois group of $\tp(b/K)$ relative to ${\mathcal Q}$
\end{Claim}

As $\tp(b/K)$ is also internal to $\mathcal Q$ and $\mathcal Q$ contains ${\mathcal C}$, it follows that the Galois group of $\tp(b/K)$ relative to ${\mathcal Q}$ is a subgroup of  $PGL(2,{\mathcal C}$). This subgroup is normal as $\mathcal Q$ is $\emptyset$-invariant. Hence, it is either trivial or the full $PGL(2, {\mathcal C})$  (by simplicity of $PGL(2,{\mathcal C})$). So let us assume that the Galois group of $\tp(b/K)$ relative to ${\mathcal Q}$ is trivial and get a contradiction.

By the remark above, this means that $(q,\pi)$ is uniformly relatively internal. So Proposition \ref{unif_int_implies_preserve} implies that $q = \tp(b,t/\mathbb{C})$ is internal to the constants. Notice that $\tp(b,t/\C)$ is isolated. So we can find a fundamental system $((b_{1}, t_{1})...,(b_{n},t_{n}))$ of realizations of $\tp(b,t)/\C)$ in the prime model $M$ (say) over $\C$ (with $(b,t) = b_{1}, t_{1})$), which therefore generates a so-called strongly normal extension $L_{1}$ of $\C$ with of course $\C\leq K\leq L \leq L_{1}$.

We will now argue that $L_{1}$ has to be a Picard-Vessiot extension of $\C$, which means precisely that $Aut_{\partial}(L_{1}/\C)$, as a (connected) algebraic group $G$  in/over $\C$ is {\em linear}. To do so, we study the structure of $L_{1}/\C$ and $G$.  Note that $t_{i}' = 1$ for $i=1,..,n$ hence $\C(t_{1},..,t_{n}) = K = \C(t)$, which we know to be a PV extension of $\C$, with Galois group $\G_{a}(\C) = (\C,+)$.  The general differential Galois theory yields a surjective homomorphism (of algebraic groups) from $G$ to $\G_{a}(\C)$, with kernel $H$ say.  Now each $K(b_{i})$ is a PV extension of $K$ with Galois group $PGL(2,\C)$ and any $\sigma\in Aut_{\partial}(L_{1},K)$ is determined by its action on each $K(b_{i})$. Hence $H$ embeds into the $n$-fold product $PGL(2,\C) \times  .. \times PGL(2,\C)$, whereby $H$ is linear. So $G$ is an extension of a linear algebraic group by a linear algebraic group so is linear.

Finally note that $G$ is not commutative as $H$ surjects onto the Galois group of $L/K$ which is $PGL(2,\C)$. On the other hand, it is well-known that the Galois group of any PV extension of an algebraically closed field of constants is commutative (in fact, with unipotent radical of dimension at most $1$).  Indeed, the PV-extension associated to $Y' = AY$ where $A$ has coefficients in $\C$ is always contained in $\mathbb{C}(t,e^{\lambda_1t}, \ldots e^{\lambda_kt})$ where $\lambda_1,\ldots, \lambda_k$ are the eigenvalues of $A$. So we get a contradiction, which proves the claim and the proposition.
\end{proof}

\subsection{$\mathrm{CCM}$} We consider now the example of $\mathrm{CCM}$, the many-sorted theory of compact complex manifolds.  A good and accessible reference is \cite{Moosa-nonstandard}.    In the standard model of $\mathrm{CCM}$ all elements of all sorts are named by constants.  $\mathrm{CCM}$ is a many-sorted superstable (in fact totally transcendental) theory of finite rank. There is a unique, up to nonorthogonality,  non locally modular strongly minimal  set, the sort ${\mathbb P}_{1}$ of the projective line over $\C$ (we will just call it $P$). It is known that $\mathrm{CCM}$ has the canonical base property \cite{Pillay-Campana}. 

In $\mathrm{CCM}$, the distinction between $\mathcal P$, $\mathcal Q$ and $\mathcal P_{int}$ collapses  over $\emptyset$:

\begin{Lemma} Suppose $T$ is stable, and $P$ is a $\emptyset$-definable set.  Work over a model $M$ (i.e.  add constants for elements of $M$).  Then $P_{int}$-definable (algebraic) coincides with $P$-definable (algebraic)
\end{Lemma}
\begin{proof} Suppose for example that $q(y)$ (over $\emptyset$, equivalently $M$) is $P_{int}$-definable. So there is $a$ realizing $q$ and $b$ with $\tp(b)\in P_{int}$ such that $a\in dcl(b)$ say $a = f(b)$ for a $\emptyset$-definable function $f$. Now for some $e$ such that $b$ is independent from $e$ over $\emptyset$, we have $b\in dcl(e,c)$ for some tuple $c$ of elements satisfying $P$.  We may assume that $a$ is independent from $e$ too. Then there is a formula $\theta(y,z)$ true of $(a,e)$ over $\emptyset$ and expressing that there are $b,c$ with $c$ a tuple of elements in $P$ and $b\in dcl(z,c)$ and $y = f(b)$.  As $\emptyset$ is a model, there is $e'\in M$ such that $\theta(a,e')$ holds, which suffices.
\end{proof}

%
%
We are led to:

\begin{Problem}\label{inverseGaloistheory-CCM} Study  inverse Galois theory relative to $P$ in the theory $\mathrm{CCM}$:  which algebraic groups can appear as the Galois group of some stationary type relative to $P$ in the theory $\mathrm{CCM}$?
\end{Problem}

We will provide a positive answer to this problem for $G = PGL(2,\mathcal C)$ whereby showing that CCM does not have the strong CBP and in the process initiate the study of Problem \ref{inverseGaloistheory-CCM} in general. 

We take as $M$ the standard model of $\mathrm{CCM}$ (where by definition all elements are named by constants).  Let ${\mathbb M}$ be a saturated elementary extension.  Let $A$ be a subset of $\mathbb M$, and $p(x)$ a complete stationary type over $A$, which is $P$-internal. So the definable Galois group of $p$ relative to $P$ is definably isomorphic to a definable group in the sort $P({\mathbb M})$ namely an algebraic group $G$ over ${\mathbb C}^{*}$, the nonstandard extension of $\C$, living in ${\mathbb M}$.  So Problem \ref{inverseGaloistheory-CCM} asks what are the possible algebraic groups which can appear. 

Note that when $A = M$ then internality of $p$ implies $P$-definability whereby the Galois group is trivial.  Now consider arbitrary $(p,A)$ where $p = \tp(a/A)$ is  stationary and $P$-internal.  As we know $\mathrm{CCM}$ is totally transcendental there is a finite tuple $b$ from $A$ such that $a$ is independent from $A$ over $b$ and $\tp(a/b)$ is stationary, so also $\tp((a,b)/b)$ is stationary and $P$-internal.  And with $q = \tp(a,b)$ and $\pi$ the projection on the second coordinate, then $(q,\pi)$ is  relatively $P$-internal.

The geometric interpretation is:  $(a,b)$ is the generic (over $M$) point of an irreducible compact complex analytic variety $X$, $b$ the generic point of  some such $Y$ and $\pi$ induces a dominant meromorphic map from $X$ to $Y$.
So in this case we have $\pi:X\to Y$ and we say that the generic fibre is internal to $P$, or that generically we have relative internality to $P$ of $\pi:X\to Y$.  Note that this is on the face of it a model-theoretic notion, as by convention we are taking the generic point $b$ of $Y$ to be living in the nonstandard model $\mathbb M$.

This suggests a couple of additional questions

\begin{Question} \label{Question-CCM} (i) Given a fibration $\pi:X\to Y$ of irreducible compact complex spaces with irreducible fibres (over a Zariski open subset of $Y$), when is the fibration generically relatively internal  to $P$?
\newline
(ii) Assuming generic relative internality, how to compute the Galois group of the generic fibre of the fibration $\pi$ relatively to $P$ (equivalently relatively to $P_{int}$)?
\end{Question}

The issue of internality to $P$ of the generic fibre is somewhat subtle and two important cases are addressed in \cite{Moosa-nonstandard}. 

\begin{itemize} 
\item[(a)] If the fibres of $\pi$ is one dimensional then we indeed have internality to $P$ of the generic fibre.
\item[(b)]  If $\pi$ is a projective bundle over a (irreducible) compact complex manifold $Y$ then we also have internality to $P$ of the generic fibre.
\end{itemize}

 By a projective bundle over a (irreducible) compact complex manifold $Y$, we mean the projectivization $P(E)\to Y$ of a holomorphic vector bundle $E$ (equivalently locally free coherent analytic sheaf) over $Y$. In case (b), it is proved in \cite{Moosa-nonstandard} that in fact a stronger form of generic relative internality holds:
 
\begin{Fact}\label{internalityofprojectivebundle} Let $P(E)\to S$ be a projective bundle where $E$ is a vector bundle of rank $n + 1$.

There is a surjective holomorphic map $\pi:T\to S$ (of compact complex manifolds), and a bimeromorphic map $g:T\times_{S}P(E) \to T\times {\mathbb P}_{n}$ over $T$, such that for general $t\in T$. $g_{t}$ induces an isomorphism (biholomorphic map) between $P(E)_{\pi(t)}$ and ${\mathbb P}_{n}$.  Here $n+1$ is the rank of the vector bundle $E$. 
\end{Fact}

The expression general means outside some countable union of proper analytic subvarieties. So this is stronger than generic which means a generic point of $S^{*}$ over $M$ (as the language is not countable). Nevertheless the conclusion of Fact \ref{internalityofprojectivebundle} implies that for generic $s\in S^{*}$ there is a isomorphism (in particular a definable bijection) between the nonstandard fibre $Y^{*}_{s}$ and ${\mathbb P}_{n}(\C^{*})$  defined over an additional parameter $t^{*}$ from $T^{*}$.  We have {\em internality} of the definable set $Y^{*}_{s}$ to the sort $P$ in a strong sense.  This begs the question of what a biholomorphic map between nonstandard compact complex manifolds is, and it means (at least) that graph of the map is Zariski closed as defined in \cite{Moosa-nonstandard}. 

As in the case of a $P$-internal complete type, we have a Galois group of a definable set $Y$ over a given base $A$ where the definable set is internal to $P$. Namely $Gal(Y/A,P)$ is the group of permutations of $Y$ which are elementary over $A, P$, equivalently the group of permutations of $Y$ induced by automorphisms of the ambient saturated model ${\mathbb M}$ which fix $A$ and $P$ pointwise. As with the case of complete types this group has the structure of an $A$-definable group with $A$ definable action, which is definably isomorphic (over additional parameters) to an algebraic group in the sense of the algebraically closed field $\C^{*}$. 

We describe some partial answers to  part (ii) of Question \ref{Question-CCM} in the case of projective bundles.  The first is:
\begin{Lemma} Suppose $Y$ is a projective algebraic variety. Then any projective bundle over $Y$ IS generically uniformly relative internal to $P$.
\end{Lemma}
\begin{proof} This follows from GAGA \cite{Serre-GAGA}, namely the fact that a projective bundle over an algebraic variety is itself an algebraic variety. 
\end{proof} 

We now describe a very different picture when $Y$ is a strongly minimal locally modular compact complex surface.  \textit{Until the end of this section, we take $Y$ to be a surface (2-dimensional compact complex manifold) strongly minimal and locally modular and $E$ to be a rank $2$ holomorphic vector bundle on $Y$. So $P(E)$ is a 3-dimensional compact complex manifold.}

 The example originates in a paper by Matei Toma \cite{Toma}.  We first give the background notions, basically using definitions in the literature.  We will take a base $Y$ to be a compact complex manifold. As mentioned above a holomorphic vector bundle $E\to Y$ is the same thing as a locally free coherent analytic sheaf on $Y$.  Let $E$ have rank $r$ (meaning that each fibre is an $r$-dimensional complex vector space.  In the literature $E$ is said to be \textit{irreducible} if $E$ has no proper coherent subsheaf of rank $r'$ with $0< r' < r$. And $E$ is said to be \textit{strongly irreducible} if for any complex compact manifold $Y'$ of the same dimension as $Y$ and holomorphic surjective map $f:Y'\to Y$, the holomorphic vector bundle $f^{*}(E)$ on $Y'$ is irreducible.

An effective divisor in $P(E)$ is a positive linear combination of analytic subvarieties of dimension $2$. It is said to be horizontal if the projection of every irreducible component of its support (under $P(E) \to Y$) is all of $Y$.

\begin{Fact}[{\citep[Proposition p.242]{Toma}}] \label{Toma-equivalence}

$E$ is strongly irreducible iff $P(E)$ has no  horizontal (effective) divisors.

\end{Fact}
Note that the right hand side above is equivalent to saying that $P(E)$ has no irreducible analytic subvariety of dimension $2$ which maps onto $Y$ (under $P(E)\to Y$).  Such a map has to of course be finite-to-one above a Zariski open subset of $Y$, by dimension considerations, so we could call it an algebraic rather than definable section

We now use that the generic type of $Y$ (over $M$) is orthogonal to $P$  (which is equivalent to non weakly orthogonal, as $M$ is a model): we consider $s$ realizing the generic type of $Y$ over $M$ and let $P(E)^{*}_{s}$ be the fibre of
$P(E)^{*} \to Y^{*}$ above $s$.  Now as $\mathrm{CCM}$ is totally transcendental, there is a prime model, say $N$, over $M,s$. The orthogonality hypothesis implies that that   $P(N) = P(M) = {\mathbb P}_{1}(\C)$. where $P$ is the sort of the projective line.  Now by Fact \ref{internalityofprojectivebundle} above there is a definable, over some parameter $t$ isomorphism between $P(E)^{*}_{s}$ and ${\mathbb P}_{1}(\C^{*})$. As $N$ is an elementary substructure of ${\mathbb M}$, we can find such $t$ in $N$, and working in $N$ gives us a $t$-definable isomorphism between $P(E)^{*}_{s}(N)$ and ${\mathbb P}_{1}(\C)$. 

\begin{Definition}
The  Galois group of $P(E)$ denoted  $G(P(E))$ is the group of elementary permutations of  $P(E)^{*}_{s}(N)$ induced by automorphisms of $N$ which fix pointwise $(M,s)$ where $s$ realizes the generic type of $Y$, $M$ is the standard model and $N$ is the prime model over $(M,s)$. 
\end{Definition}

\begin{Lemma} $G(P(E))$ is (definably) isomorphic to an algebraic subgroup $H$ of $PGL(2,\C)$.
\end{Lemma}
\begin{proof} First, we have more or less said above that $G = G(P(E))$ and its action on $P(E)^{*}_{s}(N)$ are definable over $(M,s)$ in $N$. We will write the definable, over $t$,  isomorphism, between $P(E)^{*}_{s}(N)$ and ${\mathbb P}_{1}(\C)$ as $f_{t}(-)$.  

There are some soft proofs relying on general theorems about definable groups of permutations of strongly minimal sets. But we will be more direct.  First the definable isomorphism $f_{t}(-)$ between $P(E)^{*}_{s}(N)$ and ${\mathbb P}_{1}(\C)$ yields a definable isomorphism $F$ say between $G$ and some definable group of permutations of ${\mathbb P}_{1}(\C)$. Let us show that the image of any $\sigma\in G$ is an automorphism of ${\mathbb P}_{1}(\C)$ in the sense of algebraic geometry, hence an element of $PGL(2,\C)$. 

Given $\sigma\in G$, and $c\in {\mathbb P}_{1}(\C)$, we find $F(\sigma)(c)$ by taking $f_{t}^{-1}(c)$, applying $\sigma$ then applying $f_{t}$ again, obtaining $f_{t}(f_{\sigma(t)}^{-1}(c))$. Notice that $f_{t}$ is an isomophism, as $f_{\sigma(t)}$ and its inverse, so the definition of isomorphism above shows that the composition is an automorphism of ${\mathbb P}_{1}(\C)$, as required.
Hence we get a definable isomorphism between $G$ and a definable (in $\mathrm{CCM}$) subgroup of $Aut({\mathbb P}_{1}(\C)) = PGL(2,\C)$.  As ${\mathbb P}_{1}(\C)$ with induced structure from $M$ is just the algebraically closed field $\C$ (with point at infinity), it follows that this definable subgroup is an algebraic subgroup. 
\end{proof}

\begin{Proposition} Suppose $E\to Y$ is strongly irreducible. Then  $G(P(E))$ is (definably) isomorphic to $PGL(2,\C)$.
\end{Proposition}
\begin{proof}
Let $H$ be the algebraic subgroup of $PGL(2,\C)$ isomorphic to $G$ via $F$ above.  We claim that $H = PGL(2,\C)$. If not, the connected component $H^{0}$ of $H$ is a proper connected algebraic subgroup of $PGL(2,\C)$ which is therefore contained in a conjugate of the Borel subgroup of $PGL(2,\C)$, thus is the stabilizer of some point $a$  in ${\mathbb P}_{1}(\C)$.  But then $H$ has a finite orbit in its action on  ${\mathbb P}_{1}(\C)$. 
It follows that $G$ has a finite orbit under its action on $P(E)^{*}_{s}(N)$ via automorphisms of $N$ over $(M,s)$.  As $N$ is $\omega$-homogeneous over $(M,s)$, it follows that there is $b\in P(E)^{*}_{s}(N)$, which is in $acl(M,s)$.  This yields an algebraic section of $P(E)\to Y$, which contradicts the strong irreducibility of $E$ and Fact \ref{Toma-equivalence}. 
\end{proof}

\begin{Corollary} $\mathrm{CCM}$ does not have the strong $\mathrm{CBP}$. 
\end{Corollary}

\begin{proof}
The main result of \cite{Toma} implies that one can construct a K3-surface (or complex torus)  $Y$ and a vector bundle $E$ of rank $2$ over $Y$ such that:
\begin{itemize}
\item $Y$ is strongly minimal and orthogonal to $P$,
\item $E \to Y$ is strongly irreducible.
\end{itemize}
By the proposition above, it follows that $G(P(E))$ is (definably) isomorphic to $PGL(2,\C)$.

It is now easy to find a complete (and even stationary) type which is internal to $P$ with Galois group $PGL(2,\C^{*})$:
Recall $t$ was a parameter in $N$ over which there is a definable isomorphism between  $P(E)^{*}_{s}(N)$ and ${\mathbb P}_{1}(\C)$. On general grounds  (by considering say the canonical base of $t$ over  $P(E)^{*}_{s} \cup 
{\mathbb P}_{1}(\C^{*})$) we can find a tuple ${\bar b}$ from $P(E)^{*}_{s}(N)$ such that such an isomorphism can be found definable over ${\bar b}$.  It follows that $\tp({\bar b}/M,s)$ is internal to $P$ and its Galois group is definably isomorphic to  $PGL(2,\C^{*})$. To make the type stationary, we can consider the (isolated) type of ${\bar b}$ over $(M,s)$ together with a suitable $e\in acl(M,s)$, which will not change the Galois group as it is already connected.   Hence by  Theorem \ref{strongCBP-rigidity}, CCM does not have the strong CBP.

\end{proof}

\subsection{Pairs of algebraically closed fields} In contrast, we conclude this section with a theory which does have the strong CBP for not completely trivial reasons.  We will take $T$ to be the theory of (nontrivial)  pairs of algebraically closed fields of characteristic $0$, in the language of unitary rings together with a unary predicate $P$ for the bottom field. $T$ is a reduct of $\mathrm{DCF}_0$ where the bottom field comes from the constants. It is well-known that $T$ is complete , $\omega$-stable of infinite Morley rank, and the bottom field is the only strongly minimal set, up to nonorthogonality.  Let us call the bottom field $P$. There are many definable sets in $T$ which are internal to $P$ but may need parameters to see the internality.  See \cite{Pillay-pairsACF}
for more details about this theory.  An important fact which follows from the description of independence (nonforking) in models of $T$ is that analyzability in $P$ implies internality to $P$. As we know \cite{Chatzidakis} that (in the context of superstable theories of finite rank) if $b = Cb(\stp(a/b)$ then $\stp(b/a)$ is analyzable in the family on non locally modular types of $U$-rank $1$, it follows that:
\begin{Proposition} The (finite rank part of) $T$, the theory of pairs of algebraically closed fields, has the strong canonical base property.

\end{Proposition}

\section{Some effective criteria for uniform internality in $\mathrm{DCF}_0$}
In this section, we give  effective criteria for uniform relative (almost) internality to the constants for two specific families of differential equations of the theory  $\mathrm{DCF}_0$. 

The first family has the form:
$$\begin{cases}    y' = yg(x)\\ x' = f(x) \end{cases} \text {where} f(x), g(x) \in \C(x) $$

and will be studied in Section 5.1.

The second family has the  form:
$$\begin{cases}  y' = g(x)\\x' = f(x)\end{cases} \text {where again} f(x), g(x)\in \C(x) $$
 and is mentioned in Section 5.2.

\vspace{5mm}
\noindent
Of course we are restricting $(x,y)$ to where both $f$ and $g$ are defined at $x$.  So the base and fibres are both order $1$ so strongly minimal.   And the total space is Morley rank $2$, degree $1$. As a warm up, consider the special case of the first family, where $g(x) = x$:

$$(E): \begin{cases} y' = yx  \\
x' = f(x) \end{cases} \text{ where } f(x)  \in \mathbb{C}(x).$$

Let $q$ be the unique generic type of the total space  (over $\C$), and $\pi$ the projection on $x$. So in the language of Section 3, $(q,\pi)$ is relatively internal to the constants. Moosa and Jin \cite{Moosa-Jin} work under the assumption that $\pi(q)$ is internal to the constants, and prove that $q$ is almost internal to the constants if and only if $(q,\pi)$ is almost trivial, if and only if  there exists a non-zero $k \in \mathbb{Z}$, $\alpha \in \mathbb{C}$ and $h \in \mathbb{C}(x)$ such that
$$ kx = \alpha + \mathrm{dlog}_f(h) = \alpha + f \frac {\frac d {dx} (h)} {h}$$
where $\mathrm{dlog}_f(h) = \frac{\delta_f(h)}  {h}$ is the logarithmic derivative of the derivation $\delta_f$ satisfying $\delta_f(x) = f(x)$. 

We prove a similar statement when the base $x' = f(x)$ is orthogonal to the constants replacing (almost) internality to the constants by  uniform relative (almost) internality.

\subsection{Case of the logarithmic derivative} Fix a differentially closed field $\mathcal U$ extending $\mathbb{C}$.  $\mathbb{C}\lbrace x \rbrace$ denotes the complex algebra (of infinite dimension) of differential polynomials with complex coefficients and $\mathbb{C}\langle x\rangle$ denotes its function field.

\begin{Definition}For every $g(x) \in \mathbb{C}\langle x\rangle$ written as  $g(x) = \frac{g_1(x)} {g_2(x)}$ where $g_1(x),g_2(x) \in \mathbb{C} \lbrace x \rbrace$ are two complex differential polynomials, we define:
$$ \mathcal D^{log}_g =  \lbrace (x,y) \in \mathcal U^2 \text{  } | \text{ } g_2(x) \neq 0\text{, } \delta(y) = g(x)y \text{ and } y \neq 0 \rbrace.$$ 
\end{Definition}

We also denote by $\pi$ the projection on the $x$-coordinate.  Note that $\mathcal D_g^{log}$ is a $\mathbb{C}$-definable set, its projection by $\pi$ has Morley rank $\omega$ and the fibres are strongly minimal and  internal to the constants.   In particular, for every type $q \in S(\mathbb{C})$ living on $\mathcal D^{log}_g$, $(q,\pi)$ is relatively internal to the constants.

\begin{Theorem}\label{classification-logarithmicderivative} Fix $g(x) \in \mathbb{C}\langle x\rangle$  and consider $q \in S(\mathbb{C})$ a type living on $\mathcal D^{log}_g$. If $\pi(q)$ is of finite rank and orthogonal to the constants, then the following are equivalent:
\begin{itemize}
\item[(i)] $(q,\pi)$ is uniformly relatively internal to the constants.
\item[(ii)] $(q,\pi)$ is trivial: there exists a type $p \in S(\mathbb{C})$ internal to the constants such that $q = \pi(q) \otimes p$.
\item[(iii)] For some (any) realization $b \models \pi(q)$, there exists $\alpha \in \mathbb{C}$ and a non zero $h$ in the differential field generated by $b$ and $\mathbb{C}$ such that
$$g(b) = \alpha + dlog(h)$$
where $\mathrm{dlog}: (\mathcal U,+) \rightarrow (\mathcal U,\times)$ is the usual logarithmic derivative taking $h$ to $\delta(h)/h$. 
\end{itemize}
\end{Theorem}

\begin{proof}
To show that (iii) implies (ii), we use that if $y$ is a solution of $\delta(y) = g(x)y$ then
$$dlog(y/h) = dlog(y) - dlog(h) = g(b) - dlog(h) = \alpha \in \mathbb{C}.$$

This shows that the $\mathbb{C}$-definable map $(x,y) \mapsto (x,y/h(x))$ fixes $\pi(q)$ and sends $\pi(q)$ to $q \otimes r$ where $r$ is the type of a non-zero solution of $y' = \alpha y$. (ii) implies (i) is immediate as in Proposition \ref{equivalence generale}.

It remains to show that (i) implies (iii).  We consider $x \models \pi(q)$ and let $\overline{x}$  be a finite tuple which generates  $\mathbb{C}\langle x \rangle$ as a field over $\mathbb{C}$. We denote $p = \tp(\overline{x}/\mathbb{C})$. Let $y$ be a realization of $q_x$. As $(q,\pi)$ is uniformly relatively internal, there exists $e \ind_\mathbb{C} x$ and some constants $c_1,\ldots,c_n \in \mathcal C$ such that $y \in \mathrm{dcl}(e,x,c_1,\ldots,c_n)$.

Since $\overline{x}$ and $x$ are interdefinable over $\mathbb{C}$ and $\pi(q) = \tp(x/k)$ is orthogonal to the constants so is $p$ and the condition $e \ind_\mathbb{C} x$ can be strengthened to:
$$ e,c_1,\ldots,c_n \ind_\mathbb{C} \overline{x}.$$

The differential field $\mathrm{dcl}(\C,e,c_1,\ldots, c_n,\overline{x})$ can be therefore described as the free amalgamation of the differential field $K = \mathbb{C} \langle e,c_1,\ldots, c_n \rangle$ generated by $e,c_1,\ldots c_n$ over $\mathbb{C}$ and the differential field $\mathbb{C}\langle x \rangle$ that is:
\begin{itemize}
\item as a field, $\mathrm{dcl}(\C, e, \overline{x}, c_1,\ldots c_n) = K(\overline{x}) \simeq \mathrm{Frac}(\mathbb{C}(\overline{x}) \otimes_\mathbb{C} K)$.
\item Let $D_{1}$ be the derivation  $id \otimes \delta$ on $K(\overline x)$   where $id$ is the $0$ derivation on $\C(\overline x)$,  and $D_{2}$ equal to $\delta \otimes id$ where again $id$ is $0$ on $K$. Then on $K(\overline x)$ we have

$$ \delta = D_{1} + D_{2}$$.

\end{itemize}

\begin{Claim} With the notation above, assume that the differential equation $y' = g(x)y$ admits a solution in $(K(\overline{x}),\delta)$. Then there exists $\alpha \in \mathbb{C}$ such that $y' = (g(x) - \alpha)y$ admits a solution in $(\mathbb{C}(\overline{x}),\delta)$.
\end{Claim}

\begin{proof}[Proof of the claim]
Consider the derivations $D_1,D_2$ on $K(\overline{x})$ defined above.
We have a decomposition $\delta = D_1 + D_2$ and one easily checks that the derivations $D_1$ and $D_2$ commute that is $[D_1,D_2] = 0$ as it suffices to check this identity on a system of generators and therefore on $K$ and $\mathbb{C}\langle x \rangle$ respectively. Note that this is equivalent to $[\delta,D_1] = 0$.\\

Consider now $y \in K(\overline{x})$ our solution of $y' = g(x)y$ . Since $[\delta,D_1] = 0$, we can write:
$$\delta(D_1(y)) = D_1(\delta(y)) = D_1(g(x)y) = g(x)D_1(y)$$

As the  set of solutions of the first order linear differential equation $y' = g(x)y$ in $K(\overline{x})$ is a $K(\overline{x})^\delta$- vector space of dimension at most one, we see that  $D_1(y) = \alpha y$ for some $\alpha \in K(\overline{x})^{\delta}$  (the constants of $K(\overline{x})$). But we already know that $\tp(\overline{x}/\mathbb{C})$ is orthogonal to the constants.  As $x$ is independent from $K$ over $\C$, $\tp(\overline{x}/K)$ is weakly orthogonal to the constants, so $K(\overline{x})^{\delta} = K^{\delta}$. 
So $\alpha\in K$.  We then obtain that, for some $\alpha \in K$, 
$$D_2(y) = \delta(y) - D_1(y) = g(x)y - \alpha y = (g(x) - \alpha)y.$$

We now work inside a model of $\mathrm{DCF}_0$ with a derivation extending $D_{2}$.  Note that  $\C(\overline x)$ is independent from $K$ over $\C$, and that now $K$ is a field of constants. We can write $y$ as a quotient of polynomials in elements of $\C(\overline x)$ and $K$. As the type of ${\bar x}$ over $K$ is definable over $\C$, we can find now $\alpha\in \C$ and $y\in \C(\overline x)$ such that  $D_{2}(y) = (g(x) - \alpha)y$.  But $D_{2}$ agrees with $\delta$ on $\C(\overline x)$.  This completes the proof of Claim 5.3. 
\end{proof}

To conclude the proof  of (i) implies (iii) in Theorem 5.2, note that the hypothesis of Claim 5.3 holds by construction. Hence so does the conclusion. Taking $\alpha\in C$ and $h\in \C(\overline x)$ with $h' = (g(x) - \alpha)h$ gives (iii). 
\end{proof}

\begin{Lemma}\label{almost}
For every non zero $k \in \mathbb{Z}$, denote by $m_k$ the definable map defined by $m_k(x,y) = (x,y^k)$.
\begin{itemize}
\item[(1)] If $q \in S(\mathbb{C})$ is a type living on $\mathcal D^{log}_{g(x)}$ then $q_k = (m_{k})_\ast q$ is a type living on $\mathcal D^{log}_{kg(x)}$ and $\pi(q) = \pi(q_k)$.
\item[(2)] $(q,\pi)$ is uniformly relative almost internal to the constants if and only if there exists a non-zero integer $k$ such that $(q_k,\pi)$ is uniformly relatively internal to the constants.
\end{itemize}
\end{Lemma}

Everything is obvious using that $\mathrm{dlog}(y^k) = k\mathrm{dlog}(y)$ except the direct implication of (2).  This last point follows from Example 1.17 of \cite{VanderPut-Singer}.

\begin{Corollary}\label{explicituniforminternality}  Consider a differential equation of the form:
$$(E): \begin{cases} y' = yg(x)  \\
x' = f(x) \end{cases} \text{ where } f(x),g(x)  \in \mathbb{C}(x).$$

Denote by $q$ the generic type of $(E)$ and by $\pi$ the projection on $x$.  Assume that  $\frac 1 {f(x)}$ has at least a simple pole and a multiple pole or that $\frac 1 {f(x)}$ has only simple poles but that the quotient of the residues at two of these poles is not rational. Then the following are equivalent:
\begin{itemize}
\item [(i)] $q$ is not orthogonal to the constants.
\item[(ii)]  $(q,\pi)$ is uniformly relatively almost internal.
\item[(iii)] There exists $\beta \in \mathbb{C}$ such that $ \frac {g(x) - \beta} {f(x)}$ has only simple poles with rational residues.
\end{itemize}
\end{Corollary}

\begin{proof}
First, note that the condition on $1/f(x)$ expresses that the type $\pi(q)$ (which is the generic type of $x' = f(x)$) is orthogonal to the constants (see \cite{Rosenlicht} and \citep[Example 2.20]{Hrushovski-Itai}).  The equivalence between (i) and (ii) therefore follows from Proposition \ref{equivalence generale}. To show that (ii) and (iii) are equivalent, we use Theorem \ref{classification-logarithmicderivative}:

Since $q$ lives on $\mathcal D^{log}_{g(x)}$, $(q,\pi)$ is uniformly relatively almost internal  if and only if  for some non-zero $k \in \mathbb{Z}$, $(m_k)_\ast q$ is uniformly relatively internal to the constants if and only if  there exists $\alpha \in \mathbb{C}$ and $h \in \mathbb{C}(x)$ such that $kg(x) = \alpha + f(x)  \frac {\frac{dh} {dx}} {h(x)}$ if and only if there exists $\beta \in \mathbb{C}$, a non-zero integer $k$ and $h \in \mathbb{C}(x)$ such that:
$$\frac{g(x) - \beta} {f(x)} =  \frac 1 {k} \frac {\frac{dh} {dx}} {h(x)}.$$

It is well-known that a rational function is of the form $\frac {\frac{dh} {dx}} {h(x)}$ if and only if it admits only simple poles and the residue at any of these poles is an integer (see for example \citep[Example 2.20]{Hrushovski-Itai}). The corollary follows.
\end{proof}

\begin{Example}
Consider $(E_1)$ and $(E_2)$ the differential equations given respectively by  
$$(E_1): \begin{cases} y' = yx  \\
x' = x^2(x-1) \end{cases} \text{ and } (E_2): \begin{cases} y' = yx  \\
x' = x^3(x-1) \end{cases} .$$
Denote by $q_1$ the generic type of $(E_1)$ and $q_2$ the generic type of $(E_2)$. Corollary \ref{explicituniforminternality} shows that $(q_1,\pi)$ is uniformly relatively almost internal but that $(q_2,\pi)$ is not. So $q_2$ is orthogonal to the constants and $q_1$ is not by Proposition \ref{equivalence generale}.

\end{Example}
\subsection{Case of the derivative} With appropriate modifications, the results of the previous section  hold when the logarithmic derivative is replaced by the derivative.

\begin{Definition} For every $g(x) \in \mathbb{C}\langle x\rangle$ written as  $g(x) = \frac{g_1(x)} {g_2(x)}$ where $g_1(x),g_2(x) \in \mathbb{C} \lbrace x \rbrace$ are two complex differential polynomials, we define:
$$ \mathcal D^{\delta}_g =  \lbrace (x,y) \in \mathcal U^2 \text{ }| \text{  } g_2(x) \neq 0 \text{ and } \delta(y) = g(x) \rbrace.$$ 
\end{Definition}

\begin{Proposition} Fix $g(x) \in \mathbb{C}\lbrace x\rbrace$ and consider $q \in S(\mathbb{C})$ a type living on $\mathcal D_g$. If $\pi(q)$ is of finite rank and orthogonal to the constants, then the following are equivalent:
\begin{itemize}
\item[(i)] $(q,\pi)$ is uniformly internal to the constants.
\item[(ii)] $(q,\pi)$ is trivial
\item[(iii)] For some (any) realization $b \models \pi(q)$, there exists $\alpha \in \mathbb{C}$ and a non zero $h$ of the differential field generated by $b$ and $\mathbb{C}$ such that
$$g(b) = \alpha + \delta(h).$$
\end{itemize}
\end{Proposition}

The proof (similar to the proof of Theorem \ref{classification-logarithmicderivative}) is left to the reader. Likewise the analogous statement to Corollary 5.5, as well as its proof, is left to the reader.

\begin{Example}
Consider $(E_1)$ and $(E_2)$ the differential equations given respectively by  
$$(E_1): \begin{cases} y' = x  \\
x' = x^2(x-1) \end{cases} \text{ and } (E_2): \begin{cases} y' = x  \\
x' = x^2(x-1)(x+1) \end{cases} .$$

Denote by $q_1$ the generic type of $(E_1)$ and $q_2$ the generic type of $(E_2)$. 

For $\alpha = 1$, $\frac {x - \alpha}{x^2(x-1)}$ does not have simple poles so it is of the form $\frac d  {dx} h(x)$ for some rational function $h \in \mathbb{C}(x)$ (indeed $h(x) = -1/x$ here). It follows that  $(q_1,\pi)$ is uniformly relatively almost internal, so $q_1$ is not orthogonal to the constants.

On the other hand, there are no $\beta \in \mathbb{C}$ such that $\frac {x - \beta}{x^2(x-1)(x+1)}$ does not have a simple pole. It follows that $(q_2,\pi)$ is not uniformly relatively internal so that $q_2$ is orthogonal to the constants.
\end{Example}
\section{Lifting orthogonality to the constants from a hypersurface}\label{linearization}

In this section, we use the language of $D$-varieties to formulate a non-linear version (Theorem \ref{Theorem-normal bundle}) of Proposition \ref{equivalence generale}. Recall from the preliminaries section that if $(K,\delta)$ is a differential field, one can consider the category of $D$-varieties  over $(K,\delta)$ whose objects are pairs $(X,D)$  where $X$ is an algebraic variety over $K$ and $D: \mathcal O_X \rightarrow \mathcal O_X$ is a derivation extending the derivation $\delta$ on $K$. 

We say that a closed subvariety $Z$ of $(X,D)$ is invariant  if the sheaf of ideals $\mathcal I_Z$ defining $Z$ is invariant under the derivation $D$ that is
$$D(\mathcal I_Z(U)) \subset \mathcal I_Z(U) \text{ for all open subsets }U \text{ of } X.$$

If $\mathcal U$ is a fixed differentially closed field extending $(K,\delta)$ and $(X,D)$ a $D$-variety over $(K,\delta)$, $(X,D)^\delta$ denotes the set of solutions of $(X,D)$ inside $\mathcal U$. This set can be identified with a $K$-definable subset of $\mathcal U^n$.

\subsection{Linearization along an invariant subvariety}  Let $X$ be a smooth algebraic variety over some field $k$ and $Z$ be a closed subvariety of $X$ defined by a sheaf of ideal $\mathcal I$. The conormal bundle $\mathcal N_{X/Z}^\vee$ is the kernel of the morphisms of coherent sheaves on $Z$:
$$ 
i^\ast \Omega^1_{X/k} \rightarrow  \Omega^1_{Z/k} $$
that sends a one form $\omega$ on $X$ along $Z$ to the one form $i^\ast \omega$ on $Z$.

If $f \in \mathcal O_X(U)$ is a function on an open set $U$ of $X$ vanishing on $Z$, then $i^\ast df = 0$. We therefore get a morphism $i^\ast \mathcal I \rightarrow  \mathcal N^{\vee}_{X/Z}$ whose kernel contains all the restrictions of functions $f \in \mathcal I^2(U)$ that is a morphism: 
$$i^\ast (\mathcal I / \mathcal I^2) \rightarrow \mathcal N^{\vee}_{X/Z}.$$
which sends a function $f \in \mathcal O_X(U)$ vanishing on $Z$ to the one-form $df \in \mathcal N^{\vee}_{X/Z}(U)$. When $Z$ is a smooth variety, it is well known that $\mathcal N_{X/Z}^\vee$ is a a locally free sheaf on $Z$ and that the previous morphism is an isomorphism of locally free sheaves on $Z$ (see \citep[Chapter 2, Section 8]{Hartshorne}).


\begin{Definition}
Let $X$ be a smooth algebraic variety over some field $k$ and $Z$ be a smooth closed subvariety of $X$. 
The normal bundle $\pi: N_{X/Z} \rightarrow Z$ is the vector bundle of rank  $\mathrm{codim}_X(Z)$ over $Z$ whose sheaf of sections is $\mathcal N_{X/Z}$.
\end{Definition}

The structure sheaf of $N_{X/Z}$ can be described as follows: first, for every open set $U \subset X$ and every function $f \in \mathcal I(U)$, the one form $df \in \mathcal N^{\vee}_{X/Z}(U)$ defines a function $ \overline{f}  \in \mathcal O_{N_{X/Z}}(\pi^{-1}(U \cap Z))$ defined using the adjunction:
$$\overline{f}(x,v) = df_x(v) \text{  for every  }(x,v) \in N_{X/Z}.$$
We can then cover $X$ by affine open sets $U$ such that there exist functions $f_1,\ldots f_p \in \mathcal O_X(U)$ whose images  form a basis of $\mathcal I(U)/\mathcal I^2(U)$. The structure sheaf on $N_{X/Z}$ on such opens sets $U$ is given by:
$$ \mathcal O_{N_{X/Z}}(\pi^{-1}(U \cap Z)) =  \mathrm{Sym}^\bullet_{\mathcal O_Z(U\cap Z)} (\mathcal I(U)/\mathcal I^2(U)) = \mathcal O_Z(U\cap Z)[\overline{f_1},\ldots, \overline{f_p}]. $$ 

\begin{Proposition} \label{normalbundle} Let  $(X,D)$ be a smooth $D$-variety over a differential field $(K,\delta)$ and $Z$ a smooth closed invariant subvariety of $X$. There exists a unique derivation $D_{lin}$ on the normal bundle $N_{X/Z}$ of $Z$ in $X$ such that:
\begin{itemize}
\item[(1)] The canonical projection $\pi: (N_{X/Z},D_{lin}) \rightarrow (Z,D_{Z})$ is a morphism of $D$-varieties over $(K,\delta)$.
\item[(2)] If $z \in (Z,D_Z)^\delta$, then the fibre $(N_{X/Z},D_{lin}) _z$ is a linear differential equation over the differential field $k(z)$. 
\item[(3)] For every $f \in \mathcal \mathcal \mathcal O_X(U)$ vanishing on $Z$,
$$D_{lin}(\overline{f}) = \overline{D(f)}.$$
\end{itemize}
\end{Proposition}

\begin{proof}
Uniqueness follows from the discussion above. To prove existence, note that since $D$ preserves $\mathcal I$, it also preserves $\mathcal I^2$ so the restriction of $D$ to $\mathcal I$ defines a map on the quotient
$$ \mathcal L_D: i^\ast (\mathcal I /\mathcal I^2) \rightarrow i^\ast (\mathcal I/ \mathcal I^2).$$
Denoting by $D_Z$ the derivation induced by $D$ on $\mathcal O_Z = i^\ast \mathcal O_X/\mathcal I$, an easy computation shows that for every local sections $a \in \mathcal O_X(U)$ and $m \in \mathcal I(U)$
$$ \mathcal L_D(\overline{a}.\overline{m}) = D(a)m + a.D(m)  \text{ mod } \mathcal I^2  = D_Z(\overline{a})m + \overline{a}. \mathcal L_D(m)$$  
which means that $(i^\ast (\mathcal I /\mathcal I^2),\mathcal L_D)$ is a sheaf of $D$-modules over $(\mathcal O_Z,D_Z)$. The description of the structure sheaf of $N_{X/Z}$ shows that $\mathcal L_D$ extends uniquely to a derivation $D_{lin}$ on $\mathcal O_{N_{X/Z}}$ satisfying (1),(2) and (3).
\end{proof}
\begin{Definition}
Let $Z$ be a smooth closed invariant subvariety of $(X,D)$. The $D$-variety $(N_{X/Z}, D_{lin})$ given by Proposition \ref{normalbundle} is called \textit{the first order linearization of $(X,D)$}.
\end{Definition}

Note that if we assume that $Z$ is irreducible then so is $N_{X/Z}$ (as the total space of any vector bundle over $Z$). If $q \in S(K)$ denotes the generic type of $(N_{X/Z}, D_{lin})$ then by  (2) of Proposition \ref{normalbundle}, $(q,\pi)$ is relatively internal to the constants. To study almost uniform internality, we describe more precisely the generic fibre of $\pi$ when $Z$ is an invariant hypersurface.

\begin{Fact}[\cite{VanderPut-Singer}, pp74]
For every differential field $(K,\delta)$, the isomorphism classes (modulo gauge equivalence) of one-dimensional $D$-modules over $(K,\delta)$ are classfied by the cokernel $H(K,\delta)$ of 
$$\mathrm{dlog}_\delta : \mathbb{G}_m(K) \rightarrow \mathbb{G}_a(K).$$
\end{Fact}

If $M$ is a one dimensional $D$-module over $(K,\delta)$, we write $[M] \in H(K,\delta)$ for its equivalence class.

\begin{Lemma}\label{presentation-nomarlbundle}  Let $(X,D)$ be a smooth affine $D$-variety over $(K,\delta)$ and $Z$ a closed smooth invariant hypersurface of $(X,D)$.  Denote by $t$ a generator of the ideal defining $Z$ and write  $$D(t) = ht \text{ for some function } h \in \mathcal O_X(X).$$

If $z \in \mathcal U$ realizes the generic type of $(Z,D_Z)$ over $K$, then the fibre $(N_{X/Z},D_{lin})_z$ is one dimensional $D$-module over the differential field $K(z)$ and 
$$[(\mathcal N_{X/Z},D_{lin})_z] = - h(z) + \mathrm{dlog}(K(z)) \in H(K(z),\delta).$$
\end{Lemma}

\begin{proof}
We can assume that $X$ is affine. Since the ideal $I_Z = (t)$ of $K[X]$ is generated by $t$ as an $A$-module, it follows that $\overline{t} = t/I_Z^2$ generates  $I_Z/I_Z^2$ as a $K[Z] =K[X]/I_Z$-module.  Moreover, 

So $\mathcal N_{X/C,z}^\vee$ is a one-dimensional $D$-module generated by $\overline{t}$ over $K(Z)$.  Now we have:

$$\mathcal L_D(\overline{t}) = D(t) \text{ mod } I_Z^2 = h \cdot t \text{ mod } I_Z^2  = h \cdot \overline{t}.$$ 

Let $z \in \mathcal U$ be a realization of the generic type of $(Z,D_Z)$. It follows that $(\mathcal N_{X/C}^\vee, \mathcal L_D)_z$ is isomorphic to the $D$-module $K(z).\overline{t}$ with the connection $\nabla(\overline{t}) = h(z)\cdot \overline{t}$. Hence, 
 $$[ (\mathcal N_{X/C}^\vee, \mathcal L_D)_z] =  h(z) + \mathrm{dlog}(K(z)) \text{ and } [(\mathcal N_{X/Z},D_{lin})_z] = - h(z) + \mathrm{dlog}(K(z))$$
since the two are dual $D$-modules.
%
\end{proof}
\begin{Example}
Let $v(x,y) = f(x,y) \frac \partial {\partial x} + g(x,y) \frac  {\partial} {\partial dy}$ be a complex planar polynomial vector field. The line $y = 0$ is invariant under the vector field $v$ if and only if $y$ divides $g(x,y) \in \mathbb{C}[x,y]$.

Assume that $y = 0$ is an invariant curve and write $g(x,y) = yg_1(x,y)$. The normal bundle of the line $y = 0$ is isomorphic to $\mathbb{C}^2$ and the first-order linearization along $y = 0$ is given by the vector field:
$$ v_{lin}(x,y) = f(x,0) \frac \partial {\partial x} + yg_1(x,0) \frac \partial {\partial y}.$$
\end{Example}

\subsection{Non-uniform invariant hypersurfaces} If $Z$ is a  closed invariant hypersurface of $(X,D)$, we denote by $D_Z$ the derivation induced by $D$ on $Z$.

\begin{Theorem}\label{Theorem-normal bundle}
Let $Z$ be a smooth closed invariant hypersurface of $(X,D)$. Denote by $q$ the generic type of the first order linearization $(N_{X/Z},D_{lin})$ and  by $\pi$ the projection towards $(Z,D_Z)$. Assume that:
\begin{itemize}
\item[(i)] The type $\pi(q)$ (which is the generic type of $(Z,D_Z)$) is orthogonal to the constants,
\item[(ii)] $(q,\pi)$ is not uniformly relatively almost internal to the constants.
\end{itemize}
Then the generic type of $(X,D)$ is orthogonal to the constants.
\end{Theorem}

The main tool in the proof of Theorem \ref{Theorem-normal bundle} is the discrete valuation associated to the irreducible hypersurface $Z$. We first recall its differential properties: consider $(X,D)$ a smooth irreducible algebraic variety over $(K,\delta)$ and $Z$ an invariant irreducible hypersurface of $(X,D)$. Denote by
$$v= v_Z: K(X) \setminus \lbrace 0 \rbrace \rightarrow \mathbb{Z}$$
the discrete valuation associated to $Z$. Note that
\begin{itemize}
\item the valuation ring $\mathcal O_{v}$ is the ring of rational functions $f \in K(X)$ which are well-defined at the generic point of $Z$ (i.e when $X$ is affine, the localisation of $K[X]$ at all functions which do not vanish on $Z$).
\item the maximal ideal $m_v$ is the subset of rational functions in $\mathcal O_{v}$ which are identically zero on $Z$.
\item the residue field $\mathcal O_Z/m_Z \simeq K(Z)$ is the field of rational functions on $Z$ and the residue map coincide with the restriction
$$ res: \begin{cases} \mathcal O_{v} \rightarrow K(Z) \\
f \mapsto f_{|Z} = i^\sharp(f) \end{cases}$$
where $i : Z \rightarrow X$ is the closed immersion defining $Z$.
\end{itemize}

Since derivations extend uniquely to ring  localizations, the derivation $D$ on $K[X]$ extends uniquely to a derivation (still denoted $D$) on $\mathcal O_{v}$.

\begin{Claim}\label{compatibilityvaluationderivaiton} With the notation above, the residue map $\mathrm{res}: (\mathcal O_v,D) \rightarrow (K(Z),D_Z)$ is a morphism of differential rings.
\end{Claim}

\begin{proof}
Indeed, we can write $\mathcal O_v = \bigcup_{U \in \mathcal J_Z} K[U]$ where $\mathcal J_Z$ is the set of open sets of $X$ containing the generic point of $Z$. On each of the open sets $U \in \mathcal J_Z$, the residue map coincide with the restriction map:

$$\mathrm{res}_U: \begin{cases}
 (K[U],D) \rightarrow K[U \cap Z] \subset (K(Z),D_Z)\\
 f \mapsto f_{|Z}
 \end{cases}$$ 
 
Since $U \cap Z$ is a closed invariant subvariety of $(U,D)$, the morphism $\mathrm{res}_U$ is a morphism of differential rings.
\end{proof} 

\begin{proof}[Proof of Theorem \ref{Theorem-normal bundle}] Let $(X,D)$ be a smooth $D$-variety over $(K,\delta)$ and  let $Z$  be an invariant hypersurface of $(X,D)$. Up to restricting $X$, we can always assume that $X$ is affine and that $Z$ is smooth. We consider $t$  a generator of the ideal defining $Z$ and denote by $v$ the valuation $v_Z$ associated to $Z$. Note that $t$ is a uniformizer for the discrete valuation $v$.

Now using that $Z$ is invariant, there exists a function $h \in \mathcal O_v$ such that  $D(t) = ht$. We write $\mathrm{res}(h) = h_Z \in K(Z)$. By Lemma \ref{presentation-nomarlbundle},  the fibre over a generic point of $(Z,D_Z)$ of the first order linearization is gauge equivalent to $\mathrm{dlog}(y) = -h_Z.$ So by assumption,  $(q,\pi)$ is not uniformly relatively almost internal where $q(x,y) \in S(K)$ the generic type of $$ (\ast): \begin{cases} \mathrm{dlog}(y) =  -h_{Z} \\ \overline{x} \in (Z,D_Z)^\delta \end{cases} \text{ and by } \pi \text{ the projection on }  (Z,D_Z)^\delta.$$ 
For the sake of a contradiction, assume that the generic type of $(X,D)$ is not orthogonal to the constants. After an extension of the parameters, we can assume that $(X,D)$ is not weakly orthogonal to the constants. Therefore, there exists $f \in K(X) \setminus K^\delta$ with $D(f) = 0$.  Up to replacing $f$ by $1/f$, we can assume that $v(f) \geq 0$.

\begin{itemize}
\item If $v(f) = 0$, then $\mathrm{res}(f) = f_{Z}$ is not identically zero on $Z$. Since  $(Z,D_Z)$ is orthogonal to the constants, it follows that $f_Z = c$ for some $c \in K^\delta$. By replacing $f$ by $f - c$, we can therefore assume that $v(f) > 0$.

\item If $m = v(f) > 0$ then write $f = t^m k$ with $v(k) = 0$. Then 
$$ \mathrm{dlog}(k) = \mathrm{dlog}(f) - m \mathrm{dlog}(t) =  mh.$$
\end{itemize}
Using Claim \ref{compatibilityvaluationderivaiton}, we obtain that  $\mathrm{dlog}(res(k)) = m h_Z$ which implies that the equation  $\mathrm{dlog}(y) =  m h_Z$ has a solution in the residue field $(K(Z),D_Z)$ . This shows that for the generic type $q_m$ of 

$$ \begin{cases} \mathrm{dlog}(y) =  mh_{Z} \\ \overline{x} \in (Z,D_Z)^\delta \end{cases}$$ 
is not uniformly relatively internal to the constants and therefore that $(q,\pi)$ is not uniformly relatively almost internal to the constants (see Lemma \ref{almost})
\end{proof}

\subsection{Application to autonomous differential equations} Putting together Theorem \ref{Theorem-normal bundle} and the classification of uniformly relatively internal types given by Corollary \ref{explicituniforminternality}, we get the following criterion to check orthogonality for planar algebraic vector field with an invariant line:

\begin{Corollary}
Let $v(x,y) = f(x,y) \frac \partial {\partial x} + yg(x,y) \frac  {\partial} {\partial y}$ be a complex planar polynomial vector field with $f,g \in \mathbb{C}[x,y]$.  The line $y = 0$ is invariant under the vector field $v$. Assume that:
\begin{itemize}
\item[(i)] The polynomial $f(x,0)$ is not identically zero and the rational function  $\frac 1 {f(x,0)} \in \mathbb{C}(x)$ has at least a simple pole and a multiple pole or that $\frac 1 {f(x,0)}$ has only simple poles but that the quotient of the residues at two of these poles is not rational.

\item[(ii)] There does not exist $\beta \in \mathbb{C}$ such that $ \frac {g(x,0) - \beta} {f(x,0)} \in \mathbb{C}(x)$ has only simple poles with rational residues.
\end{itemize}
Then the generic type of $(\mathbb{A}^2,v)$ is orthogonal to the constants.
\end{Corollary}

\begin{Example}
For every choice of polynomials $f_1(x), \ldots, f_n(x),g_2(x), \ldots, g_n(x) \in \mathbb{C}[x]$, the generic type of the differential equation:
$$ \begin{cases}
x' = x^3(x - 1) + f_1(x)y+ \ldots  + f_n(x)y^n \\
y' = xy + + g_2(x)y+ \ldots  + g_m(x)y^m
\end{cases}$$
is orthogonal to the constants.
\end{Example}

\section{The Kolchin tangent bundle and uniform relative internality}

As remarked in the introduction, \cite{C-H-T-M} proves that differential jet bundles, in particular differential tangent bundles preserve internality to the constants.  So one can ask also whether different tangent bundles are uniformly (almost) relatively internal to the constants. We first give some positive results and then show, by giving counterexamples, that this is not true in general, as could be expected. 

We work in a (saturated) differentially closed field $\mathcal U$ as usual, where $(K,\delta)$ is a differential subfield.  Kolchin defined tangent spaces (and hence bundles) of arbitrary differential algebraic varieties ($DAV$'s) (over $K$). In \cite{Pillay-Ziegler} it was pointed out how these can obtained in the case where the $DAV$ is of the form $(X,D)^{\delta}$ where $(X,D)$ is an algebraic $D$-variety over $K$.  This is summarized below. 

\begin{Lemma}\label{presentation-tangentbundle}
Let $(X,D)$ be a smooth $D$-variety over $(K,\delta)$. There exists a unique structure of $D$-variety $D^1$ on $T_{X/K}$ such that:
\begin{itemize}
\item[(1)] The canonical projection $\pi: (T_{X/K},D^1) \rightarrow (X,D)$ is a morphism of $D$-varieties over $(K,\delta)$.
\item[(2)] If $z \in (X,D)^\delta$, the fibre $(T_{X/K},D_1)^{\delta}_z$ is a linear differential equation over the differential field $K(z)$.
\item[(3)] If $f \in \mathcal O_X(U)$, denote by $\overline{f} \in \mathcal O_{T_{X/K}}(\pi^{-1}(U))$ the function defined by $df$ on $\pi^{-1}(U)$. For every $f \in \mathcal O_X(U)$,
$$ D^1(\overline{f}) =  \overline{D(f)}.$$
\item[(4)] If $f : (X,D_X) \rightarrow (Y,D_Y)$ is a morphism of $D$-varieties over $(K,\delta)$, there is a morphism $d^1f$ of $D$-varieties making the following diagram commute:

\begin{tikzcd}
(T_{X/K},D^1_X) \arrow[d, "\pi"] \arrow[r,"d^1f"] & (T_{Y/K},D^1_Y) \arrow[d, "\pi"]\\
(X,D_X) \arrow[r,"f"] & (Y,D_Y)
\end{tikzcd}

\noindent making $(X,D) \rightarrow (T_{X/K},D_1)$ into a product-preserving covariant functor on the category of $D$-varieties. 

\item[(5)] If $f : (X,D_X) \rightarrow (Y,D_Y)$ is a morphism of $D$-varieties over $(K,\delta)$, then for any $b \in X$ generic, the map $d^1f_b : T_{X/K,b} \rightarrow T_{Y/K,f(b)}$ is linear. Moreover, if $f$ is generically finite-to-one, then $d^1f_b$ is an isomorphism. 

\end{itemize}
\end{Lemma}

Let $(X,D)$ be a smooth $D$-variety over $(K,\delta)$.  The $D$-variety $(T_{X/K},D^1)$ is called the Kolchin-tangent bundle of $(X,D)$. Assume that $X$ is irreducible. Then $T_{X/K}$ is irreducible and we denote by $q_{(X,D)}$ the generic type of $(T_{X/K},D^1)^{\delta}$. By Lemma \ref{presentation-tangentbundle}, $(q_{(X,D)},\pi)$ is relatively internal to the constants and $\pi(q_{(X,D)})$ is the generic type of $(X,D)^{\delta}$. We consider the property:
$$ (\ast): (q_{(X,D)},\pi)  \text{ is uniformly relatively almost internal to the constants.}$$

\begin{Proposition}\label{uniform-internality-of-tangent}
 Let $(X,D)$ be a smooth and irreducible $D$-variety over $(K,\delta)$. 
\begin{itemize}
\item[(i)] If the generic type of $(X,D)^{\delta}$ is internal to the constants then $(X,D)$ satisfies $(\ast)$.
\item[(ii)] If $\mathrm{dim}(X) = 1$ and $(X,D)$ is autonomous then $(X,D)$ satisfies $(\ast)$.
\item[(iii)] If $(X_1,D_1)$ and $(X_2,D_2)$ are two $D$-varieties satisfying $(\ast)$ then $(X_1,D_1) \times (X_2,D_2)$ also satisfies $(\ast)$.
\item[(iv)] If $(X_1,D_1)$ and  $(X_2,D_2)$ are two $D$-varieties in generically finite to finite correspondence then $(X_1,D_1)$ satisfies $(\ast)$ if and only if $(X_2,D_2)$ satisfies $(\ast)$ 
\item[(v)] If $(X,D)$ is a $D$-group, then $(X,D)$ satisfies $(\ast)$.
\end{itemize}

\begin{proof}

(iii) and (iv): applications of Lemma \ref{presentation-tangentbundle} (4).

(i): By results of Chatzidakis-Moosa-Trainor, if the generic type of $(X,D)^{\delta}$ is internal to the constants, then so is the generic type of $(T_{X/K},D^1_X)^{\delta}$. Thus $(X,D)$ satisfies $(\ast)$.

(ii): So by assumption $X$ is over a field of constants. There are two cases. First, if $D = 0$ then  constants, then (i) yields that $X$ satisfies $(\ast)$.

Otherwise, recall that by construction, if $(X,D)$ is an autonomous, irreducible and smooth $D$-variety, there is a morphism of $D$-varieties $s : (X,D) \rightarrow (T_{X/K},D^1_X)$ such that $\pi \circ s = \mathrm{id}$. As $\mathrm{dim}(X) = 1$, the Kolchin tangent space $(T_{X/K},D_X^1)^{\delta}_a$ over any generic $b \in X$ is a one-dimensional $\C$-vector space. As $D$ is not $0$ the section $s$ is not trivial at a generic point $b$, i.e. $s(b) \neq (b,0)$. It thus yields a basis of $(T_{X/K},D_X^1)_b$, uniformly in $b$, which is enough to obtain uniform internality to the constants.

(v): Let $\mu : X\times X \rightarrow X$ be the group operation. By Lemma \ref{presentation-tangentbundle}, there is a morphism of $D$-varieties $d^1\mu : (T_{X/K},D^1_X) \times (T_{X/K},D^1_X) \rightarrow (T_{X/K},D^1_X)$ such that for all $a_1,_2 \in (T_{X/K},D^1_X)$, we have $\pi(d^1\mu(a_1,a_2)) = \mu(\pi(a_1),\pi(a_2))$. Moreover, by functoriality and preservation of products, this is a group operation on $(T_{X/K},D^1_X)$. 

The section $s : (X,D) \rightarrow (T_{X/K},D^1_X)$ defined by $s(b) = (b,0)$ thus is a $D$-group embedding, and we will denote $\cdot$ the group operation on $(X,D)$ and $(T_{X/K},D^1_X)$. 

Fix some $b \in (X,D)^\delta$ and a $\mathcal{C}$-basis $a_1, \cdots, a_n$ of $(T_{X/K},D^1_X)^\delta_b$. We will show that for any $(c,a) \in (T_{X/K},D^1_X)^\delta$, we have $a \in \mathrm{dcl}(c,a_1 \cdots, a_n, b, \mathcal{C})$, which yield uniform internality to the constants. 

Observe that by Lemma \ref{presentation-tangentbundle}, multiplication by $c\cdot b^{-1}$ defines a linear isomorphism from $(T_{X/K},D^1_X)^\delta_b$ to $(T_{X/K},D^1_X)^\delta_c$. In particular, the tuple $(c \cdot b^{-1} \cdot a_1, \cdots , c \cdot b^{-1}\cdot a_n)$ is a $\mathcal{C}$-basis of $(T_{X/K},D^1_X)^\delta_c$, and is contained in $\mathrm{dcl}(b,c,a_1, \cdots a_n)$. Moreover, as it is a basis, we also have that $a \in \mathrm{dcl}(c \cdot b^{-1} \cdot a_1, \cdots , c \cdot b^{-1}\cdot a_n, \mathcal C)$. Thus we finally obtain $a \in \mathrm{dcl}(c,a_1 \cdots, a_n, b, \mathcal{C})$. 
\end{proof}

\end{Proposition}
\subsection{A counterexample}
\begin{Proposition}\label{counterexample} Consider the Kolchin-closed set defined by 
$$(E): \begin{cases} x' = x^3(x-1) \\
y' = xy + \frac {y^2} 2.
\end{cases} $$
Then the Kolchin tangent bundle of $(E)$ is not uniformly relatively almost internal to the constants.
\end{Proposition}

Geometrically, $(E)$ is the differential equation associated with the planar vector field $v(x,y) = x^3(x-1)\frac \partial {\partial x} + (xy + y^2) \frac {\partial} {\partial y}$. The line $y = 0$ is invariant, the generic type of this line is orthogonal to the constants and the linearization of $(E)$ along $y = 0$ is the equation:
$$(E_{lin}): \begin{cases} x' = x^3(x-1) \\
y' = xy.
\end{cases}.$$
We have seen that $(q,\pi)$ is not uniformly relative almost internal and that therefore the generic type of $(E)$ is orthogonal to the constants.

Moreover, the projection $\pi$ of $(E)$ on $x' = x^3(x-1)$ defines a rational factor and therefore an invariant foliation on $(\mathbb{A}^2,v)$ which is the foliation generated by $w = \frac d {dy}$. A direct computation shows that:

$$ \mathcal L_v(w) = [v,w] = (x + y)w.$$ 

Consider the differential equation
$$(T_{red}): \begin{cases}
x' = x^3(x-1) \\
y' = xy + \frac {y^2} 2 \\
z' = (x + y)z
\end{cases}$$

\begin{Claim}\label{counterexampleclaim1}
With the notation above, denote by $q$ the generic type of the Kolchin tangent space of $(E)$ and by $q_{red}$ the generic type of $T_{red}$. Then:
$$q = q_{red} \otimes p_\mathcal C.$$
where $p_\mathcal C$ denotes the generic type of the field of constants.
\end{Claim}

\begin{proof}
We denote by $X_0$ the open subset of $\mathbb{A}^2$ on which the vector fields $v$ and $w$ are transversal (in particular, not zero). We get a trivialization of the tangent bundle of $X_0$ as: 
$$ \Phi: \begin{cases} X_0 \times \mathbb{A}^2 \rightarrow T_{X_0/\mathbb{C}} \\
(x,y, \lambda,\mu) \mapsto (x,y, \lambda v(x,y) + \mu w(x,y)) \end{cases}$$

We want to compute the derivation $D^1$ in the chart given by $\Phi$. Note that since the first projection is a morphism of $D$-varieties the derivation $D^1$ read in the chart given by $\Phi$ is determined by the values of $D^1(\lambda^\ast)$, $D^1(\mu^\ast)$ where $(\lambda^\ast,\mu^\ast)$ are the coordinate functions. We consider 
$$\omega_1 = dx/f(x) \text{ and } \omega_2 = dy - \frac{g(x,y)}{f(x)} dx = dy - g(x,y) \omega_1$$
where $f(x) = x^3(x - 1)$ and $g(x,y) = xy + y^2/2$. It is the dual basis of the basis given by the vector fields $v$ and $w$ so $\lambda^\ast = \omega_1 \circ \Phi$ and  $\mu^\ast = \omega_2 \circ \Phi$. An easy computation using Lemma \ref{presentation-tangentbundle} shows that: 

$$ D^1(\omega_1) = 0 \text { and } D^1(\omega_2) = \frac {\partial g} {\partial y}(x,y) \omega_2 = (x + y)\omega_2. $$

It follows that the derivation $D^1$ read in the chart $\Phi$ is given by the vector field:

$$v_1(x,y,\lambda, \mu) = v(x,y) + 0 \cdot \frac \partial {\partial \lambda} + (x + y)\mu \frac{ \partial} {\partial \mu} $$
which implies that $(X_0 \times \mathbb{A}^2  ,v_1)  \simeq (X_0 \times \mathbb{A}_1, w_1) \times (\mathbb{A}^1,0)$ where $w_1(x,y,\mu) = v(x,y) + (x + y)\mu \frac{ \partial} {\partial \mu}$ and therefore that $q$ is interdefinable with $q_{red} \otimes p_\mathcal{C}$.
\end{proof}

\begin{Claim}\label{counterexampleclaim2}
Let $(x,y)$ be a generic solution of $(E)$ then for all non zero  $k \in \mathbb{Z}$, all $c \in \mathbb{C}$ and all $h \in \mathbb{C}(x,y)$, 
$$ky \neq c + \mathrm{dlog}_\delta(h)$$
where $\delta$ is the unique derivation trivial on $\mathbb{C}$ such that $\delta(x) = f(x)$ and $\delta(y) = xy + y^2/2$.
\end{Claim}

\begin{proof}
We use the invariant line $y = 0$ and the valuation  $v$ attached to this invariant hypersurface. For the sake of a contradiction, assume that we can find a non zero  $k \in \mathbb{Z}$ , a non zero $h \in \mathbb{C}(x,y)$ and $c \in \mathbb{C}$ such that 
$$ (\ast): kyh = ch +\delta(h).$$

We write $n = v(h)$ and $h = y^n u$ for some unit $u \in \mathcal O_v^\ast$. The identity $(\ast)$ gives:
$$ ky^{n+1}u= c y^n u + y^n \delta(u) + ny^{n}x u,$$

and after dividing by $y^n$, we obtain:
$$ \delta(u) = -(c + nx)u + kyu.$$
and using that the residue map is a morphism of differential rings (Claim \ref{compatibilityvaluationderivaiton}):
$$ \delta_C(res(u)) = - (c + nx)res(u)$$
where $\delta_C = f(x) \frac d {dx}$ is the derivation induced by $D$ on the line $y = 0$. It follows that:
\begin{itemize}
\item either $n = 0$ and the equation $y' = -cy$ has a non-zero solution in $(\mathbb{C}(x),f(x) \frac d {dx})$. This is impossible since $(\mathbb{A}^1,f(x) \frac d {dx})$ is orthogonal to the constants and therefore to the generic type of  $y' = -cy$.

\item either $n \neq 0$ and then $nx =c + \mathrm{dlog}(res(u))$ which implies by Corollary \ref{explicituniforminternality} that $(E_{lin})$ is not uniformly relatively internal to the constants. \qedhere
\end{itemize}
\end{proof}

\begin{proof}[Proof of Proposition \ref{counterexample}]
By Claim \ref{counterexampleclaim1}, we have that $q = q_{red} \otimes p_\mathcal C$. It is therefore enough to prove that $(q_{red},\pi)$ is not uniformly relatively almost internal to the constants.

We now compute the image of $q_{red}$ under the change of variables $(x,y,z) \mapsto (x,y, z/y)$. If $(x,y,z) \models q_{red}$ then 
$$ \mathrm{dlog}(z/y) = \mathrm{dlog}(z) - \mathrm{dlog}(y) = x + y - x = y.$$

Therefore, $q_{red}$ is interdefinable with the generic type $ q_2$ of 
$$(T^2_{red}): \begin{cases}
x' = x^3(x-1) \\
y' = xy + \frac {y^2} 2 \\
w' = yw
\end{cases}$$

We now conclude that $(q_2, \pi)$ is not uniformly relatively almost internal to the constants: by Theorem \ref{explicituniforminternality}, otherwise there exist a non zero  $k \in \mathbb{Z}$ , a non zero $h \in \mathbb{C}(x,y)$ and $c \in \mathbb{C}$ such that 
$ ky = c +\mathrm{dlog}(h).$ This contradicts Claim \ref{counterexampleclaim2}. It follows that $(q_{red}, \pi)$ and $(q,\pi)$ are not uniformly relatively almost internal to the constants
\end{proof}

\bibliographystyle{plain}
\bibliography{bibliographie}

\end{document}